\documentclass[a4paper]{amsart}
\usepackage{amsmath,amsthm,amssymb,latexsym,epic,bbm,comment,mathrsfs}
\usepackage{graphicx,enumerate,stmaryrd,xcolor}
\usepackage[all,2cell,knot]{xy}
\xyoption{2cell}

\newtheorem{theorem}{Theorem}
\newtheorem{lemma}[theorem]{Lemma}
\newtheorem{corollary}[theorem]{Corollary}
\newtheorem{proposition}[theorem]{Proposition}

\newtheorem{remark}[theorem]{Remark}
\usepackage[all]{xy}
\usepackage[active]{srcltx}
\usepackage[parfill]{parskip}
\usepackage{enumerate}

\newcommand{\tto}{\twoheadrightarrow}
\font\sc=rsfs10

\newcommand{\cP}{\sc\mbox{P}\hspace{1.0pt}}

\begin{document}
\title[Categorification of the desingularization map]
{Finite quotients of singular Artin monoids and \\ categorification of the desingularization map}

\author[Jonsson, Mazorchuk, Persson Westin, Srivastava, Stroi{\'n}ski, Zhu]
{Helena Jonsson, Volodymyr Mazorchuk, Elin Persson Westin,\\
Shraddha Srivastava, Mateusz Stroi{\'n}ski and  Xiaoyu Zhu}

\begin{abstract}
We study various aspects of the structure and representation
theory of singular Artin monoids. This includes a number of 
generalizations  of the desingularization map and explicit 
presentations for certain finite quotient monoids of diagrammatic
nature. The main result is a categorification of the classical
desingularization map for singular Artin monoids associated to 
finite Weyl groups using BGG category $\mathcal{O}$.
\end{abstract}

\maketitle

\section{Introduction and description of the results}\label{s1}

\subsection{Brief introduction}\label{s1.1}

Singular braid monoids are certain natural generalizations of braid groups
that play an important role in the theory of finite type invariant of
knots and links (a.k.a. Vassiliev invariants), see \cite{Bi93}.
A singular braid allows for two strands to cross each other, creating
a so-called {\em singular crossing}. A singular link is the closure of 
a singular braid and to study invariants of singular links is a
natural and interesting problem. A singular braid monoid admits 
a presentation which generalizes Artin's presentation for the 
corresponding braid group, see \cite[Lemma~3]{Bi93}. Extrapolating
to the setup of an arbitrary Coxeter matrix allows one to define
a singular Artin monoid, for each Artin braid group.

There is an interesting relation between a singular Artin monoid
and the corresponding (regular) Artin braid group.
Interpreting a singular crossing as the difference between 
the left and right crossings defines a homomorphism from a singular
Artin monoid to the integral group algebra of the corresponding Artin
braid group. In \cite{Pa04}, it is shown that this map is 
injective for all singular braid monoids. A similar result is known
for some singular Artin monoids, see \cite{GP05,Ea06}, however,
to the best of our knowledge, the general case is still open.
Various structure properties of singular Artin monoids
and their applications to low dimensional topology were studied
by many authors during the last 30 years, see e.g. 
\cite{An04,An05,Ba92,Bi93,Co00,Co05,Ea06,Ea15,GP05,Ja17,Pa04,Ve10} 
and references therein.

Each braid group admits an interesting upgrade to a $2$-category, see,
for example, \cite{Ro06}. One of the classical $2$-representations of
this object, usually referred to as a $2$-braid group, 
is given by the action of derived shuffling functors on the 
blocks of the Bernstein-Gelfand-Gelfand (BGG) category $\mathcal{O}$
for the Lie algebra $\mathfrak{sl}_n$, see \cite{MaStr}.
It is even known that this action is faithful, see \cite{KS02}.
In \cite[Theorem~1]{MaStr2} (see also \cite[Theorem~5.1]{CM}), 
this was extended to a categorical action of the singular braid monoid.
Unfortunately, this construction does not have any obvious connection
to the desingularization map, even on the level of the Grothendieck group.

\subsection{Story of the paper: how and why}\label{s1.2}

During the academic year 2020-2021, the second author gave a PhD level
course about category $\mathcal{O}$ at Uppsala University
(all lectures are available on YouTube via the Uppsala Algebra Channel).
One of these lectures was exactly about the categorical action of the
singular braid monoid on category $\mathcal{O}$. The research presented
in this paper started from some questions asked by some participants
of the course after that lecture. One of the main questions was whether
one could define another categorical action of the singular braid monoid
on category $\mathcal{O}$ that would be better related to the classical
desingularization map. The discussions that started during the spring of 
2021 were organized as a study group on singular Artin monoids during 
the academic year 2021-2022 in Uppsala. The present paper is the report, 
by all participants, of the discussions during the meetings of this 
study group. 

Now we will briefly describe what we discussed. This will
also help to summarize the results and the structure of the present paper.

We looked into three main directions:
\begin{itemize}
\item Generalizations of the desingularization map.
\item Finite diagrammatic quotients of singular Artin monoids.
\item Categorification of the desingularization map.
\end{itemize}

One could argue that some of the results of the paper could
be cleaned up, strengthened and pushed further. However, the
time constraints on the organization of this 
study group put strict bounds on what we could include in this paper.

\subsection{Results: generalized desingularization map}\label{s1.3}

The first of the research directions, on various generalizations of the
desingularization map, was originally motivated by 
the construction in \cite[Theorem~1]{MaStr2}.
If the latter result is not related to the classical desingularization map,
then what kind of map is it related to?  In Section~\ref{s3} we
investigate various generalizations of the desingularization map.
In particular, in Subsection~\ref{s3.2} we consider the odd skeleton
of a Coxeter graph, that is, the graph obtained from the Coxeter graph 
by removing all even edges. We show that one can define an
analogue of the desingularization map for each assignment of a Laurent 
polynomial to a connected component of the odd skeleton, see 
Proposition~\ref{prop2}. Moreover, for all cases for which the
classical desingularization map is known to be injective, we establish,
in Proposition~\ref{prop3}, the injectivity of a generic choice for 
the generalized desingularization map constructed in Proposition~\ref{prop2}.
In Subsection~\ref{s3.2}, we also discuss some further generalizations of
the classical desingularization map in the setup when one replaces
the integral group algebra of the Artin braid group by the semialgebra
of the Artin braid group over the Boolean semiring. 

The situation encountered in \cite[Theorem~1]{MaStr2} is explained 
in Section~\ref{s9}. The group algebra of the braid group has an
interesting quotient, called the Hecke algebra. The latter controls
the combinatorics of category $\mathcal{O}$. The categorification of the
singular braid monoid proposed in \cite[Theorem~1]{MaStr2} corresponds to
a generalized desingularization map for this quotient. 
This latter map is given by assigning
to the connected components of the odd skeleton of the corresponding 
Coxeter graph the Laurent polynomial describing the Kazhdan-Lusztig basis
element of the Hecke algebra associated to a simple reflection.

\subsection{Results: finite diagrammatic quotients of singular Artin monoids}\label{s1.4}

The second major direction which we explore studies various finite 
diagrammatic quotients of singular Artin monoids. Diagram algebras
are intensively studied in modern representation theory due to their
numerous applications in various areas of contemporary mathematics
and theoretical physics. Such algebras usually have a basis given by
some kind of combinatorially defined diagrams, and multiplication is
based on concatenation of these diagrams. The most general such object
(which contains all other as subobjects) is the algebra of partitioned
binary relations, studied in \cite{MM13}. Other examples include the 
Brauer algebra, see \cite{Br}, and the partition algebra, see \cite{Mar94}. 

Diagram algebras are usually defined as deformations of the corresponding
diagram monoid. In the latter, multiplication is given exactly by
concatenation of diagrams, which usually involves some straightening
procedure. A diagram monoid usually has the symmetric group as the
group of invertible elements. The symmetric group is, of course,
a quotient of the braid group. It is therefore natural to ask
whether the canonical projection from the braid group to the symmetric
group can be extended to a homomorphism from the singular braid monoid
to the diagram monoid in question. In case such an extension is possible,
one could try to characterize its kernel, leading to a presentation of
the image by generators and relations. As a next step, based on this
presentation, one could then try to define analogues of this diagram
monoid for other Coxeter types. That was our main idea for this second
direction addressed in the paper.

All results related to finite quotients of singular Artin monoids
are collected in Section~\ref{s4}. As it turned out, most of them 
can be organized so that one starts with a maximal possible quotient
and then, step by step, projects further.

Our starting point is a homomorphism from the singular braid monoid
to the monoid of all binary relations, see Proposition~\ref{prop-4.7.2-1}.
The homomorphism itself is inspired by our generalized desingularization
maps. In this case, instead of the usual integral group algebra
of the braid group, we consider, as the target of our homomorphism,
the semialgebra of the braid group over the 
Boolean semiring. The image of the singular generator $\tau_s$ under this
map should be thought of as the set $\{e,\sigma_s\}$ consisting of the
identity $e$ and the corresponding regular element $\sigma_s$. This 
map appeared previously in the context of $0$-Hecke monoids, see
\cite{MaSt}. The monoid generated by the images of the singular elements
(in type $A$) was studied in \cite{MaSt} under the name of double Catalan
monoid. The image of the singular braid monoid under this map is also
related to the factor power of the symmetric group as defined and studied in
\cite{GM3,GM4}. In Subsections~\ref{s4.7.3} and \ref{s4.7.4} we explore
type $B$ analogues of these results.

The next target for a homomorphism is the dual symmetric inverse 
monoid $\mathbf{I}^*_n$, that is, the inverse monoid of all bijections 
between quotients of the finite set $\underline{n}=\{1,2,\dots,n\}$.
In Proposition~\ref{prop4.3.1-1} we show that this monoid is a natural
quotient of the monoid of essential binary relations on $\underline{n}$.
The latter contains the image of the singular braid monoid. 
Composing with the projection map onto the quotient, we get the induced
homomorphism from the singular braid monoid to $\mathbf{I}^*_n$.
The image of this homomorphism is the maximal factorizable submonoid 
$\mathbf{F}^*_n$ of $\mathbf{I}^*_n$. The monoid $\mathbf{F}^*_n$ 
was studied in detail by several authors. In particular, there are 
various presentations of this monoid which can be found in the literature,
see \cite{Ea06b,EEF,Fi03}.

In Subsections~\ref{s4.3.4} and \ref{s4.3.5} we look at certain type $B$
analogues of the monoids $\mathbf{F}^*_n$ and $\mathbf{I}^*_n$.
In particular, in Proposition~\ref{prop-4.3.7-1}, by examining the
kernel of the natural projection from the type $B$ singular Artin monoid,
we obtain  a presentation for the type $B$ analogue of the monoid
$\mathbf{F}^*_n$.

We also note that the congruence describing the kernel of the map
from the essential binary relations to  $\mathbf{I}^*_n$ can be
extended to all binary relations. The corresponding quotient is
the monoid of all bijections between quotients of subsets of
$\underline{n}$, studied in \cite{KMU}.

The next stop is the classical symmetric inverse monoid 
$\mathbf{IS}_n$ of all bijections between subsets of $\underline{n}$. 
It turns out, see Proposition~\ref{lem-isn-dual}, that restricting elements of 
$\mathbf{F}^*_n$ to singletons defines a homomorphism to $\mathbf{IS}_n$.
This leads to a natural homomorphism from the singular braid 
monoid to $\mathbf{IS}_n$. This homomorphism is not surjective,
its image ``misses'' all elements with singleton defect.
Again, there is the obvious analogue of the story in type $B$,
explored in Subsections~\ref{s4.2.4}~-- \ref{s4.2.6}.
This includes a presentation for the signed rook monoid
in Proposition~\ref{prop-4.2.6-1}, see also \cite{EF13} and \cite{CE}
for further generalizations.

Finally, in Subsection~\ref{s4.4} we study natural surjections from singular
braid monoids to Brauer monoids. The type $A$ story is fairly
expected, see Proposition~\ref{prop-4.4.2-1}. We also consider
the natural type $B$ analogue of this story, which turns out to be
different from another type $B$ analogue of Brauer algebras considered
in \cite{CL}. As explained in Subsection~\ref{s4.4.65}, our type $B$ 
Brauer monoid is related to the partial Brauer monoid introduced in \cite{Ma98}.
In fact, our type $B$ Brauer monoid is a partially $\mathbb{Z}_2$-colored version
of the partial Brauer monoid. In Theorem~\ref{prop-4.4.6-1} we extend
the presentation of singular braid monoids and type $A$ Brauer monoids to
a presentation of type $B$ Brauer monoids.

\subsection{Results: categorification of the desingularization map}\label{s1.5}

Our final main direction describes a new categorical action of 
singular Artin monoids on category $\mathcal{O}$. All necessary preliminaries 
on category $\mathcal{O}$ are collected in Section~\ref{s5}.

The classical realization of the generators of the braid group as
endofunctors of category $\mathcal{O}$ uses derived (co)shuffling
functors introduced in \cite{Ca}. In \cite[Theorem~1]{MaStr2} it is
proposed to realize the action of the singular crossings using
indecomposable projective functors (associated to the corresponding
simple root). All this is recalled in Section~\ref{s6}.
The most difficult part of the proof of this results is to check
the mixed braid relations between the singular and the regular
generators. In type $A$, the corresponding argument in 
the proof of \cite[Theorem~1]{MaStr2} is incomplete, a complete
proof is given in \cite[Theorem~5.1]{CM} based on an alternative
approach outlined in \cite{MaStr2}. In Theorem~\ref{thm6.2.1} we
give a fairly short general proof of these mixed braid relations
in any type. In Proposition~\ref{prop6.1.1} we also give a general
proof of braid relations for (co)shuffling functors in all types.

Our final main results are in Section~\ref{s7}. To make a
connection with the desingularization map, we observe that there
are two obvious natural transformations between the derived
shuffling and coshuffling functors. The usual philosophy of 
categorification suggests that the singular crossing should
act via the cone of one of these. And we, indeed, show that 
one of the cones works, see Theorem~\ref{thm7.2.1}, while the other one
fails, see Proposition~\ref{prop7.6-3}. The proof of 
Theorem~\ref{thm7.2.1} is split into verifying the defining
relations for the generators of the singular braid monoid, 
one at a time.

\subsection{Structure: preliminaries}\label{s1.6}

Additionally to the above, in Section~\ref{s2} below we collected all
necessary preliminaries, including the definitions for all main protagonists
of this paper which are illustrated by a detailed example in type $A$.

\subsection*{Acknowledgments}
This research is partially supported by the Swedish Research Council.
We thank Hankyung Ko for very helpful discussions. 
We thank James East for helpful comments.

\section{Singular Artin monoids}\label{s2}


\subsection{Coxeter groups}\label{s2.2}

Let $S$ be a finite set with $n>0$ elements.
Consider an $n\times n$ matrix $M=(m_{s,t})_{s,t\in S}$ with entries in 
$\mathbb{Z}_{>0}\cup\{\infty\}$ satisfying the following conditions:
\begin{itemize}
\item $m_{s,s}=1$, for all $s\in S$,
\item $m_{s,t}=m_{t,s}$, for all $s,t\in S$,
\item $m_{s,t}\neq 1$ provided that $s\neq t$.
\end{itemize}
Such a matrix $M$ is called a {\em Coxeter matrix}. Associated to $M$, we
have the corresponding {\em Coxeter group} $W=W(M)$ generated by $S$ subject to
the relations
\begin{displaymath}
(st)^{m_{s,t}}=e,\quad
\text{ for all }s,t\in S\text{ such that }m_{s,t}\neq\infty.
\end{displaymath}
Note that, for $s=t$, we have the relation $s^2=e$. Taking this into account, 
we have the following alternative set of relations defining $W$:
\begin{displaymath}
\begin{array}{rcll}
s^2&=&e,& \text{ for all }s\in S;\\
\underbrace{ststs\dots}_{m_{s,t}\text{ factors}}
&=&\underbrace{tstst\dots}_{m_{t,s}\text{ factors}},&
\text{ for all }s\neq t\in S\text{ such that }m_{s,t}\neq \infty.
\end{array}
\end{displaymath}

The same information as given by matrix $M$ can be described by the
corresponding Coxeter graph $\Gamma_M$ defined as follows:
\begin{itemize}
\item the set of vertices of $\Gamma_M$ is $S$;
\item $\Gamma_M$ has no loops;
\item there is an (unoriented) edge between two different vertices
$s$ and $t$ in $\Gamma_M$ if and only if $m_{s,t}>2$;
\item in case $m_{s,t}>3$, this edge is marked by $m_{s,t}$.
\end{itemize}

\subsection{Artin groups}\label{s2.3}

The Artin group (a.k.a. Artin-Tits group) associated to $W$ is the
group $\mathbf{B}(W)$ generated by $\sigma_s$, where $s\in S$,  
subject to the relations
\begin{displaymath}
\underbrace{\sigma_s\sigma_t\sigma_s\dots}_{m_{s,t}\text{ factors}}
=\underbrace{\sigma_t\sigma_s\sigma_t\dots}_{m_{t,s}\text{ factors}},\qquad
\text{ for all }s\neq t\in S\text{ such that }m_{s,t}\neq \infty.
\end{displaymath}

\subsection{Singular Artin monoids}\label{s2.4}

The singular Artin monoid associated to $W$ is the
monoid $\mathbf{SB}(W)$ generated by $\sigma_s$ and $\tau_s$, where $s\in S$,  
subject to the relations
\begin{eqnarray}\label{sbm-eq1}
\underbrace{\sigma_s\sigma_t\sigma_s\dots}_{m_{s,t}\text{ factors}}
&=&\underbrace{\sigma_t\sigma_s\sigma_t\dots}_{m_{t,s}\text{ factors}},\qquad
\text{ for all }s\neq t\in S\text{ s. t. }m_{s,t}\neq \infty;\\
\label{sbm-eq2}
\underbrace{\tau_s\sigma_t\sigma_s\dots \sigma_t\sigma_s}_{m_{s,t}\text{ factors}}
&=&\underbrace{\sigma_t\sigma_s\sigma_t\dots\sigma_s\tau_t}_{m_{t,s}\text{ factors}},
\text{ for all }s\neq t\in S\text{ s. t. }m_{s,t}\neq \infty\text{ is odd};\\
\label{sbm-eq3}
\underbrace{\tau_s\sigma_t\sigma_s\dots \sigma_s\sigma_t}_{m_{s,t}\text{ factors}}
&=&\underbrace{\sigma_t\sigma_s\sigma_t\dots\sigma_t\tau_s}_{m_{t,s}\text{ factors}},
\text{ for all }s\neq t\in S\text{ s. t. }m_{s,t}\neq \infty\text{ is even};\\
\label{sbm-eq4}
\tau_s\tau_t&=&\tau_t\tau_s,\qquad\qquad\quad m_{s,t}=2;\\
\label{sbm-eq5}
\tau_s\sigma_s&=&\sigma_s\tau_s.
\end{eqnarray}

\subsection{Special case: singular braid monoid}\label{s2.5}

In the special case when we take $S=\{s_1,s_2,\dots,s_{k-1}\}$,  for some $k\geq 2$, 
and the matrix
\begin{displaymath}
M=\left(
\begin{array}{cccccccc}
1&3&2&2&\dots&2&2&2\\
3&1&3&2&\dots&2&2&2\\
2&3&1&3&\dots&2&2&2\\
2&2&3&1&\dots&2&2&2\\
\vdots&\vdots&\vdots&\vdots&\ddots&\vdots&\vdots&\vdots\\
2&2&2&2&\dots&1&3&2\\
2&2&2&2&\dots&3&1&3\\
2&2&2&2&\dots&2&3&1\\
\end{array}
\right),
\end{displaymath}
which corresponds to the following Coxeter graph:
\begin{equation}\label{eq1}
\xymatrix{s_1\ar@{-}[r]&s_2\ar@{-}[r]&s_3\ar@{-}[r]&\dots\ar@{-}[r]&s_{k-1}}, 
\end{equation}
we obtain the following classical objects.

The associated Coxeter group is isomorphic to the symmetric group $S_k$,
via the isomorphism which sends $s_i$ to the transposition $(i,i+1)$.

The associated Artin group is isomorphic to the braid group on $k$ strands
via the isomorphism which sends $\sigma_i:=\sigma_{s_i}$ to the following braid:
\begin{displaymath}
\xymatrix@C=3mm{
1\ar@{-}[d]&2\ar@{-}[d]&\dots&i-1\ar@{-}[d]&
i\ar@{-}[dr]&i+1\ar@{-}[dl]|\hole&i+2\ar@{-}[d]&\dots&k\ar@{-}[d]\\
1&2&\dots&i-1&i&i+1&i+2&\dots&k
}
\end{displaymath}

The associated singular Artin monoid is isomorphic to the singular braid monoid on $k$ 
strands via the isomorphism which sends $\sigma_i:=\sigma_{s_i}$ to the above braid
and which sends the generator $\tau_i:=\tau_{s_i}$ to the following singular braid:
\begin{displaymath}
\xymatrix@C=3mm{
1\ar@{-}[d]&2\ar@{-}[d]&\dots&i-1\ar@{-}[d]&
i\ar@{-}[dr]&i+1\ar@{-}[dl]|-{\bullet}\ar@{-}[dl]&i+2\ar@{-}[d]&\dots&k\ar@{-}[d]\\
1&2&\dots&i-1&i&i+1&i+2&\dots&k
}
\end{displaymath}

\section{Desingularization map and its generalizations}\label{s3}

\subsection{Desingularization map}\label{s3.1}

Let $M$ be a Coxeter matrix with the associated Coxeter group $W$. 
Consider the integral group algebra 
$\mathbb{Z}(\mathbf{B}(W))$. This algebra is, in particular, a
monoid with respect to multiplication.

\begin{lemma}\label{lem1}
There is a monoid homomorphism $\Delta:\mathbf{SB}(W)\to \mathbb{Z}(\mathbf{B}(W))$ 
which is the identity on all $\sigma_s$ and sends each $\tau_s$
to $\sigma_s-\sigma_s^{-1}$.
\end{lemma}

\begin{proof}
It is straightforward to verify that the images in $\mathbb{Z}(\mathbf{B}(W))$ 
of the generators of $\mathbf{SB}(W)$ under $\Delta$ satisfy the defining
relations of $\mathbf{SB}(W)$.
\end{proof}

The map $\Delta$ was originally considered in \cite{Ba92}. It
is called the {\em desingularization map}, see \cite{Bi93}.
For the classical singular braid monoid, the map $\Delta$ has the following
topological interpretation:
\begin{displaymath}
\xymatrix@C=3mm@R=2mm{
\ar@{-}[ddrr]&&\ar@{-}[ddll]\ar@{-}[ddll]&&&&
\ar@{-}[ddrr]&&\ar@{-}[ddll]|\hole&&&&\ar@{-}[ddrr]|\hole&&\ar@{-}[ddll]\\
&\bullet&&&=&&&&&&-&&&&\\
&&&&&&&&&&&&&&
}
\end{displaymath}
In \cite{Pa04}, it is shown that, in the case of the classical singular
braid monoid, the map $\Delta$ is injective. Similar results are also known
in a number of other cases, see \cite{GP05,Ea06}.

\subsection{The odd skeleton}\label{s3.2}

We want to generalize the desingularization map. In order to give 
the full generalization, we need to introduce the notion of the
{\em odd skeleton} of a Coxeter graph.

Let $\Gamma$ be a Coxeter graph. The {\em odd skeleton} $\mathbf{O}(\Gamma)$
of $\Gamma$ is the graph obtained from $\Gamma$ by removing all edges that 
are marked by even numbers or by $\infty$. First, to mention a trivial example, 
if  $\Gamma$ is the graph given by \eqref{eq1}, then 
$\mathbf{O}(\Gamma)=\Gamma$. At the same time, in type $B_4$, we have:
\begin{displaymath}
\Gamma:\xymatrix{\bullet\ar@{-}[r]^4&\bullet\ar@{-}[r]&\bullet\ar@{-}[r]&\bullet},
\qquad
\mathbf{O}(\Gamma):
\xymatrix{\bullet&\bullet\ar@{-}[r]&\bullet\ar@{-}[r]&\bullet}.
\end{displaymath}

\subsection{Generalized desingularization map}\label{s3.3}

Let $M$ be a Coxeter matrix with the corresponding Coxeter group $W$. 
Let $\Gamma_M$ be the corresponding Coxeter graph
and $K$ the set of connected components of $\mathbf{O}(\Gamma)$.
For $s\in S$, we denote by $K_s$ the connected component of $\mathbf{O}(\Gamma)$
containing $s$. 

For a fixed map $\Phi:K\to \mathbb{Z}[x,x^{-1}]$, define
\begin{displaymath}
\Delta_\Phi(\sigma_s)=\sigma_s\text{ and }
\Delta_\Phi(\tau_s)=\Phi(K_s)[\sigma_s],\text{ for }s\in S.
\end{displaymath}

\begin{proposition}\label{prop2}
The above assignment extends uniquely to a monoid homomorphism
$\Delta_\Phi:\mathbf{SB}(W)\to \mathbb{Z}(\mathbf{B}(W))$.
\end{proposition}

\begin{proof}
We need to check that the images of the generators of $\mathbf{SB}(W)$
in $\mathbb{Z}(\mathbf{B}(W))$ under $\Delta_\Phi$
satisfy the defining relations of $\mathbf{SB}(W)$.
The relations in \eqref{sbm-eq1} are obvious. The relations in 
\eqref{sbm-eq4} follow from the relations 
$\sigma_s\sigma_t=\sigma_t\sigma_s$ provided that $m_{s,t}=2$,
which, in particular, implies that $\sigma_s^i\sigma_t^j=\sigma_t^j\sigma_s^i$
provided that $m_{s,t}=2$, for all $i,j\in\mathbb{Z}$.

The relations in \eqref{sbm-eq3} follow similarly from the 
corresponding relations for $\mathbf{B}(W)$. The latter, basically, say
that $\sigma_s$ commutes with $\sigma_t\sigma_s\dots \sigma_t$ 
(the latter word is of length $m_{s,t}-1$) and hence
$\sigma_s^i$ commutes with $\sigma_t\sigma_s\dots \sigma_t$, for all
$i\in \mathbb{Z}$.

Next, to check the relations in \eqref{sbm-eq2}, we will use
that the map $\Phi(K_{{}_-}):S\to \mathbb{Z}[x,x^{-1}]$ 
is constant on the connected components of 
the odd skeleton. The relations in $\mathbf{B}(W)$ say that moving
$\sigma_s$ past $\sigma_t\sigma_s\dots \sigma_s$ (of length $m_{s,t}-1$)
to the right, we get $\sigma_t$. This implies a similar claim for
$\sigma_s^i$ and $\sigma^i_t$, for any $i\in \mathbb{Z}$. Now, taking into
account that the Laurent polynomials assigned to $s$ and $t$ by $\Phi$
are the same (since $m_{s,t}$ is odd), the necessary relation follows.

Finally, the relations in \eqref{sbm-eq5} follow from the fact that 
$\sigma_s$ commutes with all its integer powers.
\end{proof}

It is natural to call $\Delta_\Phi$ a {\em generalized desingularization map}.

\subsection{Generic generalized desingularization map}\label{s3.4}

For a fixed $k\in\mathbb{Z}_{\geq 0}$, denote by $\mathcal{A}_k$
the set of all possible choices of $\Phi:K\to \mathbb{Z}[x,x^{-1}]$
such that $\Phi(K_s)$ is of the form
\begin{displaymath}
a_{-k}x^{-k}+a_{-(k-1)}x^{-(k-1)}+\dots+a_{k-1}x^{k-1}+a_{k}x^{k},
\text{ for some }  a_{-k},\dots, a_{k}\in\mathbb{Z},
\end{displaymath}
for any simple reflection $s$. The set $\mathcal{A}_k$ has the obvious structure
of a free $\mathbb{Z}$-module of rank $(2k+1)^{|K|}$.

As usual, we will say that a subset of $\mathcal{A}_k$ is {\em generic}
provided that this subset is dense in $\mathbb{C}\otimes_{\mathbb{Z}}\mathcal{A}_k$
with respect to Zariski topology. The next proposition shows that, in many cases,
``almost all'' choices for the generalized desingularization map lead to 
injective monoid homomorphisms.

\begin{proposition}\label{prop3}
Assume that the Coxeter group $W$ belongs to one of the following cases:
\begin{enumerate}[$($a$)$]
\item \label{prop3.1} the symmetric group (i.e. type $A$),
\item \label{prop3.2} the dihedral group (i.e. type $I_2(n)$),
\item \label{prop3.3} the right-angled group (i.e. each $m_{s,t}\in\{2,\infty\}$,
for all $s\neq t$).
\end{enumerate}
Then, for each $k\in\mathbb{Z}_{\geq 1}$, the set of all $\Phi\in\mathcal{A}_k$,
for which $\Delta_\Phi$ is injective, is generic.
\end{proposition}

\begin{proof}
We start by extending the scalars from $\mathbb{Z}$ to $\mathbb{C}$.
We consider the group algebra $\mathbb{C}(\mathbf{B}(W))$ and, for any
choice of $\Phi:K\to \mathbb{C}[x,x^{-1}]$, the corresponding 
monoid homomorphism $\Delta_\Phi:\mathbf{SB}(W)\to \mathbb{C}(\mathbf{B}(W))$.
Then we also have the corresponding set $\mathcal{A}_k^\mathbb{C}$ which is
naturally isomorphic to $\mathbb{C}\otimes_{\mathbb{Z}}\mathcal{A}_k$ as a
vector space.

Let $\Phi\in\mathcal{A}_k^\mathbb{C}$ and $u$ and $v$ be two different
elements in $\mathbf{SB}(W)$. Then the fact that $\Delta_\Phi(u)=\Delta_\Phi(v)$
means that the coefficients of $\Delta_\Phi(u)$ with respect to the standard basis
of $\mathbb{C}(\mathbf{B}(W))$ coincide with the corresponding coefficients
of $\Delta_\Phi(v)$. Each such coefficient is, by construction, a polynomial
in the coefficients of all $\Phi(K_s)$. This means that the condition
$\Delta_\Phi(u)=\Delta_\Phi(v)$ defines a closed subset of $\mathcal{A}_k^\mathbb{C}$.
Taking the intersection over all $u$ and $v$, we obtain that the set $\mathcal{B}$ 
of all $\Phi\in\mathcal{A}_k^\mathbb{C}$, for which $\Delta_\Phi$ is not injective,
is closed. 

Under our assumptions on the type of $W$, the fact that the set $\mathcal{B}$
does not coincide with $\mathcal{A}_k^\mathbb{C}$ follows from
the results of \cite{Pa04,GP05,Ea06} because the classical desingularization map
corresponds to the choice of a particular element in $\mathcal{A}_1\subset\mathcal{A}_k\subset
\mathcal{A}_k^\mathbb{C}$. Since $\mathcal{A}_k$
is dense in $\mathcal{A}_k^\mathbb{C}$, it follows that 
$\mathcal{A}_k\setminus \mathcal{B}$ is dense in $\mathcal{A}_k^\mathbb{C}$
as well. This completes the proof.
\end{proof}

\subsection{Further generalizations}\label{s3.5}

The content of this subsection is inspired by 
the results of \cite[Subsection~2.2]{MaSt}.
Let $R=(R,+,\cdot,0,1)$ be a commutative semiring (e.g. in the sense of 
\cite[Subsection~2.1]{CDJMPZZ}). Then we can consider 
the corresponding semialgebra $R(\mathbf{B}(W))$ and, for any
choice of $\Phi:K\to R[x,x^{-1}]$, from the obvious analogue of
Proposition~\ref{prop2}, we obtain the corresponding 
monoid homomorphism $$\Delta_\Phi:\mathbf{SB}(W)\to R(\mathbf{B}(W)).$$

One particular choice of $R$ leads to a very natural construction
and interpretation. Consider as $R$ the Boolean semiring
$\mathtt{B}=\{0:=\mathtt{false},1:=\mathtt{true}\}$ with respect
to the usual operations $+:=\vee$ and $\cdot:=\wedge$.
Then the semialgebra $\mathtt{B}(\mathbf{B}(W))$ can be naturally 
identified with the set of all finite subsets of $\mathbf{B}(W)$.
Similarly, we can view $\mathtt{B}(\mathbb{Z})$
as the set of all finite subsets of $\mathbb{Z}$.
In this case, the obvious analogue of Proposition~\ref{prop2}
can be reformulated as follows:

\begin{proposition}\label{prop4} 
For a fixed map $\Phi:K\to \mathtt{B}(\mathbb{Z})$, define
\begin{displaymath}
\Delta_\Phi(\sigma_s)=\sigma_s\text{ and }
\Delta_\Phi(\tau_s)=\{\sigma_s^i\,:\,i\in \Phi(K_s)\},\text{ for }s\in S.
\end{displaymath}
This assignment extends uniquely to a monoid homomorphism
\begin{displaymath}
\Delta_\Phi:\mathbf{SB}(W)\to \mathtt{B}(\mathbf{B}(W)). 
\end{displaymath}
\end{proposition}

We do not know whether any of these $\Delta_\Phi$'s is injective or not.

\section{Singular Artin monoids and Hecke algebras}\label{s9}

\subsection{Hecke algebra}\label{s9.1}

Let $M$ be a Coxeter matrix and $(W,S)$ the corresponding Coxeter system. 

Consider the Laurent polynomial  ring
$\mathbb{A}:=\mathbb{Z}[v,v^{-1}]$ and the corresponding group algebra
$\mathbb{A}(\mathbf{B}(W))$. Denote by $\mathbf{H}=\mathbf{H}(W)$
the quotient of $\mathbb{A}(\mathbf{B}(W))$ modulo the ideal
generated by the following relations:
\begin{displaymath}
(\sigma_s-v)(\sigma_s+v^{-1})=0,\quad \text{ for }\quad s\in S.
\end{displaymath}
Traditionally, the image of $\sigma_s$  in $\mathbf{H}$ is
denoted by $H_s$.

For $w\in W$ with a fixed reduced expression $w=s_1s_2\dots s_k$,
set $H_w:=H_{s_1}H_{s_2}\dots H_{s_k}$. Then 
$\{H_w\,:\, w\in W\}$ is an $\mathbb{A}$-basis of $\mathbf{H}$,
called the {\em standard basis}. 

Note that each $H_s$ is invertible in $\mathbf{H}$ with inverse 
$H_s^{-1}=H_s-(v-v^{-1})H_e$.
The algebra $\mathbf{H}$ admits a unique involution $\overline{\hspace{1mm}\cdot\hspace{1mm}}$, 
called the {\em bar-involution}, satisfying $\overline{H_s}=H_s^{-1}$
and $\overline{v}=v^{-1}$. Furthermore, $\mathbf{H}$ has a unique basis
$\{\underline{H}_w\,:\,  w\in W\}$ such that, for $w\in W$, we have:
\begin{itemize}
\item $\displaystyle\underline{H}_w\in H_w+\sum_{x\in W}v\mathbb{Z}[v]H_x$;
\item $\overline{\underline{H}_w}=\underline{H}_w$.
\end{itemize}
This basis is called the {\em Kazhdan-Lusztig (KL) basis}, see \cite{KL}
and note that we are using the normalization of \cite{So2}.

\subsection{Monoid homomorphisms from $\mathbf{SB}(W)$ to $\mathbf{H}$}\label{s9.2}

Fix a map $\Phi:K\to \mathbb{A}[x,x^{-1}]$. Define
\begin{displaymath}
\Upsilon_\Phi(\sigma_s)=H_s\text{ and }
\Upsilon_\Phi(\tau_s)=\Phi(K_s)[H_s],\text{ for }s\in S.
\end{displaymath}

\begin{corollary}\label{corn9.2.1}
The above assignment extends uniquely to a monoid homomorphism
$$\Delta_\Phi:\mathbf{SB}(W)\to \mathbf{H}(W).$$
\end{corollary}

A very special case of the above construction is the case when $\Phi$ sends each $K_s$
to $v+x$. In this case $\Upsilon_\Phi(\tau_s)=\underline{H}_s$.

We note that the restriction of $\Delta_\Phi$ to $\mathbf{B}(W)$ is not
faithful,  in general.

\section{Diagram algebras and singular Artin monoids}\label{s4}

In this section we investigate relations between singular Artin monoids and various 
families of diagram algebras.

\subsection{The group algebra of the Coxeter group}\label{s4.1}

\subsubsection{Connection via a generalized desingularization map}\label{s4.1.1}

Let $M$ be a Coxeter matrix with the associated Coxeter group $W$, 
the Artin braid group $\mathbf{B}(W)$ and the singular Artin monoid $\mathbf{SB}(W)$.
For a fixed map $\Phi:K\to \mathbb{Z}[x,x^{-1}]$, we have the monoid homomorphism
\begin{displaymath}
\mathbf{SB}(W)\to \mathbb{Z}(\mathbf{B}(W))\tto \mathbb{Z}(W), 
\end{displaymath}
where the first map is $\Delta_\Phi$ and the second map is the natural epimorphism.
This induces an algebra epimorphism $\overline{\Delta}_\Phi:\mathbb{Z}(\mathbf{SB}(W))\to 
\mathbb{Z}(W)$.

Since $\overline{\Delta}_\Phi(\sigma_s)=\overline{\Delta}_\Phi(\sigma_s^{-1})$, the number
of independent parameters in this construction is $2\cdot |K|$ as, for each 
element of $K$, without loss of generality, we may assume that its image under $\Phi$ 
is of the form $a+bx$.

\subsubsection{Pulling back Coxeter group modules}\label{s4.1.2}

Pulling back via $\overline{\Delta}_\Phi$ defines a functor from $W$-mod to $\mathbf{SB}(W)$-mod.
Generically, this gives a $2\cdot |K|$-parametric family of lifts to $\mathbf{SB}(W)$ of
simple $W$-modules. 

We note that, for some modules, the number of essential parameters might 
be lower. For example, this happens if the action of $s$ and $e$ on such a module are linearly dependent.
This is true, for example, for the trivial and the sign $W$-modules. In these cases we only have 
a $|K|$-parametric family of different lifts.

\subsection{Binary relations}\label{s4.7}

\subsubsection{The semigroup of binary relations}\label{s4.7.1}

For $n\in\mathbb{Z}_{\geq 1}$, set  $\underline{n}:=\{1,2,\dots,n\}$. 
Consider the semigroup $\mathbf{Bin}_n$ of all binary relations on $\underline{n}$,
see e.g. \cite{PW70}.
Binary relations on $\underline{n}$ are in obvious bijection with $n\times n$ 
matrices over the boolean semiring $\mathtt{B}$ and the semigroup structure is given by the 
usual matrix multiplication. We can alternatively view a binary relation on
$\underline{n}$ as a bipartite graph whose both parts are given by $\underline{n}$.
Here is an example:
\begin{displaymath}
\begin{array}{ccc}
{}^{\left(\begin{array}{cccc}1&0&0&1\\0&0&1&1\\1&1&0&0\\0&0&1&1\\\end{array}\right)}& 
\quad {}^{\leftrightarrow}\quad       &
{}^{{}_{\xymatrix{1\ar@{-}[d]\ar@{-}[drr]&2\ar@{-}[dr]&3\ar@{-}[dl]\ar@{-}[dr]&4\ar@{-}[dlll]\ar@{-}[dll]\ar@{-}[d]
\\1&2&3&4}}}
\end{array}
\end{displaymath}
For a binary relation $\rho$, we denote the corresponding bipartite graph by 
$\Gamma_\rho$.

The symmetric group  $S_n$ is the group of invertible elements in $\mathbf{Bin}_n$.
As usual, we denote by $s_i$ the transposition $(i,i+1)\in S_n$, where
$i=1,2,\dots,n-1$.

Denote by $\mathbf{Bin}_n^\mathrm{ess}$ the subsemigroup of $\mathbf{Bin}_n$ consisting of
all  {\em essential} binary relations, that is, $\rho\in \mathbf{Bin}_n$ satisfying the 
conditions that, for any $x\in\underline{n}$, there are $y,z\in\underline{n}$
such that $(x,y)\in\rho$  and $(z,x)\in\rho$. In the matrix language, these are 
those boolean matrices in which each row and each column is non-zero.

\subsubsection{Map from the singular braid monoid}\label{s4.7.2}

For $i=1,2,\dots,n-1$, we denote by $\mathbf{s}_i$ the binary relation which is obtained 
from the equality relation by adding $(i,i+1)$ and $(i+1,i)$.
For example,  here are the graphs of the elements $\mathbf{s}_1$, $\mathbf{s}_2$
and $\mathbf{s}_3$ in $\mathbf{Bin}_4$:
\begin{displaymath}
\xymatrix@C=2mm@R=2mm{
1\ar@{-}[d]\ar@{-}[dr]&2\ar@{-}[d]\ar@{-}[dl]&3\ar@{-}[d]&4\ar@{-}[d]\\
1&2&3&4}\qquad\qquad
\xymatrix@C=2mm@R=2mm{
1\ar@{-}[d]&2\ar@{-}[d]\ar@{-}[dr]&3\ar@{-}[d]\ar@{-}[dl]&4\ar@{-}[d]\\
1&2&3&4}\qquad\qquad
\xymatrix@C=2mm@R=2mm{
1\ar@{-}[d]&2\ar@{-}[d]&3\ar@{-}[d]\ar@{-}[dr]&4\ar@{-}[d]\ar@{-}[dl]\\
1&2&3&4}
\end{displaymath}

\begin{proposition}\label{prop-4.7.2-1}
There is a unique homomorphism $\eta:\mathbf{SB}(S_n)\to\mathbf{Bin}_n$ such that
\begin{displaymath}
\eta(\sigma_{s_i})=s_i\text{ and }
\eta(\tau_{s_i})=\mathbf{s}_i,\text{ for }i=1,2,\dots,n-1.
\end{displaymath}
\end{proposition}

\begin{proof}
This follows from the straightforward verification of the analogues  of the relations 
\eqref{sbm-eq1}, \eqref{sbm-eq2}, \eqref{sbm-eq3}, \eqref{sbm-eq4} and \eqref{sbm-eq5} for the 
elements $s_i$  and $\mathbf{s}_i$.  
\end{proof}

The submonoid of $\mathbf{Bin}_n$ generated by the elements $\mathbf{s}_i$ was studied in
\cite{MaSt} under the name {\em double Catalan monoid}. The image of $\eta$ belongs to
the factor power $\mathcal{FP}(S_n)$ of $S_n$ introduced in \cite{GM2}
and studied in \cite{GM3,GM4,GM5,Ma00,Ma09,Hu10}.

Note that the image of $\eta$ is contained in $\mathbf{Bin}_n^\mathrm{ess}$.
Denote by $\overline{\eta}$ the restriction of $\eta$ to the codomain 
$\mathbf{Bin}_n^\mathrm{ess}$.

\subsubsection{Type $B$ binary relations}\label{s4.7.3}

In the rest of this section,  we will often denote $-x$ by $\overline{x}$.
In particular, $\overline{\overline{x}}=x$.

For $n\in\mathbb{Z}_{\geq 1}$, set  $\underline{\overline{n}}:=\{\overline{1},\overline{2},\dots,\overline{n}\}$. 
Consider the monoid $\mathbf{Bin}^{(2)}_n$ of all binary relations 
on $\underline{\overline{n}}\cup\underline{n}$.
The identity in $\mathbf{Bin}^{(2)}_n$ is the diagonal binary relation 
$\Delta=\{(i,i)\,:\, i\in \underline{\overline{n}}\cup\underline{n}\}$.
Consider the anti-diagonal binary relation
$\nabla=\{(i,\overline{i})\,:\, i\in \underline{\overline{n}}\cup\underline{n}\}$
and denote by $\mathbf{Bin}_n^B$ the centralizer of $\nabla$ in $\mathbf{Bin}^{(2)}_n$.
As we will see later, this is an appropriate Coxeter type $B$ analogue of 
the semigroup $\mathbf{Bin}_n$.

It is convenient to think about the elements in $\mathbf{Bin}^{(2)}_n$
as $2n\times 2n$ boolean matrices whose columns are indexed by
$\overline{n},\overline{n-1},\dots,n$ left-to-right and whose rows are indexed by
$\overline{n},\overline{n-1},\dots,n$ top-to-bottom. In this realization, the elements 
of $\mathbf{Bin}_n^B$ are exactly the matrices which are invariant
under the $180^\circ$ rotation with respect to the center of the matrix.

In terms  of $\Gamma_\rho$, if both parts of $\Gamma_\rho$ are given by
$\overline{n},\overline{n-1},\dots,n$ left-to-right, then $\rho\in\mathbf{Bin}_n^B$ if and only
if $\Gamma_\rho$ is invariant under the flip which swaps $i$ and $\overline{i}$
in both parts.

The group of invertible elements in $\mathbf{Bin}_n^B$ is naturally isomorphic
to the group $\{\pm 1\}\wr S_n$ of signed permutations. The latter is
a standard realization of the Coxeter group of type $B_n$.

\subsubsection{Map from a type $B$ singular Artin monoid}\label{s4.7.4}

Let $W$ be the Coxeter group corresponding to the following (type $B$) Coxeter graph:
\begin{displaymath}
\xymatrix{s_0\ar@{-}[r]^4&s_1\ar@{-}[r]&s_2\ar@{-}[r]&\dots\ar@{-}[r]&s_{n-1}} 
\end{displaymath}

Denote by $\tilde{s}_0$ the element
of $\mathbf{Bin}_n^B$ which is obtained 
from the equality relation by removing $(1,1)$ and $(\overline{1},\overline{1})$
and adding $(1,\overline{1})$, $(\overline{1},1)$. In other words,
this is the transposition of $1$ and $\overline{1}$.
Further, for all $i=1,2,\dots,n-1$, we denote by $\tilde{s}_i$ the 
transposition of $i$ and $i+1$ and, simultaneously, of
$\overline{i}$ and $\overline{i+1}$. Then
$s_i\mapsto \tilde{s}_i$, for $i=1,2,\dots,n-1$, gives rise to a
natural monomorphism from $S_n$ to $\{\pm 1\}\wr S_n$.

For $i=1,2,\dots,n-1$, we denote by $\tilde{\mathbf{s}}_i$ the element
of $\mathbf{Bin}_n^B$ which is obtained 
from the equality relation by adding $(i,i+1)$, $(i+1,i)$, 
$(\overline{i},\overline{i+1})$ and $(\overline{i+1},\overline{i})$.
We also denote by $\tilde{\mathbf{s}}_0$ the 
element of $\mathbf{Bin}_n^B$ which is obtained 
from the equality relation by adding $(1,\overline{1})$ and $(\overline{1},1)$.
For example,  here are the graphs of the elements $\tilde{\mathbf{s}}_0$
and $\tilde{\mathbf{s}}_1$ in $\mathbf{Bin}_2^B$:
\begin{displaymath}
\xymatrix@C=2mm@R=2mm{
\overline{2}\ar@{-}[d]&\overline{1}\ar@{-}[d]\ar@{-}[dr]&1\ar@{-}[d]\ar@{-}[dl]&2\ar@{-}[d]\\
\overline{2}&\overline{1}&1&2}\qquad\qquad\qquad
\xymatrix@C=2mm@R=2mm{
\overline{2}\ar@{-}[dr]\ar@{-}[d]&\overline{1}\ar@{-}[d]\ar@{-}[dl]
&1\ar@{-}[d]\ar@{-}[dr]&2\ar@{-}[d]\ar@{-}[dl]\\
\overline{2}&\overline{1}&1&2}
\end{displaymath}

\begin{proposition}\label{prop-4.7.4-1}
There is a unique homomorphism $\boldsymbol{\eta}:\mathbf{SB}(W)\to\mathbf{Bin}_n^B$
such that
\begin{displaymath}
\boldsymbol{\eta}(\sigma_{s_i})=\tilde{s}_i\text{ and }
\boldsymbol{\eta}(\tau_{s_i})=\tilde{\mathbf{s}}_i,\text{ for }i=0,1,\dots,n-1.
\end{displaymath}
\end{proposition}

\begin{proof}
This follows from the straightforward verification of the analogues  of the relations 
\eqref{sbm-eq1}, \eqref{sbm-eq2}, \eqref{sbm-eq3}, \eqref{sbm-eq4} and \eqref{sbm-eq5} for the 
elements $\tilde{s}_i$  and $\tilde{\mathbf{s}}_i$.  
\end{proof}

The submonoid of $\mathbf{Bin}_n^B$ generated by the elements $\tilde{\mathbf{s}}_i$ 
is a natural type $B$ analogue of the {\em double Catalan monoid},
see also \cite[Subsection~6.2]{MaSt}.

\subsection{Dual symmetric inverse monoid}\label{s4.3}

\subsubsection{Dual symmetric inverse monoid and its factorizable submonoid}\label{s4.3.1}

For a positive integer $n$, consider
the symmetric inverse monoid $\mathbf{IS}_n$, defined as the set of 
all bijections between subsets of $\underline{n}$. 
We will talk about $\mathbf{IS}_n$ in more detail in Subsection~\ref{s4.2}.
In this subsection we will focus on the dual  object, called the 
{\em dual symmetric inverse monoid}  $\mathbf{I}^*_n$  which is defined as the 
set of all bijections between quotients of $\underline{n}:=\{1,2,\dots,n\}$, see \cite{FL06}.
The reason why we start with $\mathbf{I}^*_n$ will be explained in Lemma~\ref{lem-isn-dual}.

We can view each element of $\mathbf{I}^*_n$ as a binary relation on 
$\underline{n}$ in the natural way. That is, given two set partitions
\begin{displaymath}
A_1\coprod A_2\coprod\dots \coprod A_k=
\underline{n}=B_1\coprod B_2\coprod \dots\coprod B_k
\end{displaymath}
and $\sigma\in S_k$, the (unique) element $\xi$ of $\mathbf{I}^*_n$ which sends 
$A_i$ to $B_{\sigma(i)}$, for $i=1,2,\dots,k$, corresponds to the binary relation 
\begin{displaymath}
\rho_{{}_\xi}:=\bigcup_{i=1}^k  (B_{\sigma(i)}\times A_i).
\end{displaymath}
Note that $\rho_{{}_\xi}$ is essential.
The bipartite graph $\Gamma_{\rho_{{}_\xi}}$ has the property that each connected
component of $\Gamma_{\rho_{{}_\xi}}$ is a complete bipartite graph. Conversely,
given an essential binary relation $\rho$ such that each connected component of 
$\Gamma_{\rho}$ is a complete bipartite graph, there exists $\xi\in \mathbf{I}^*_n$
such that $\rho=\rho_{{}_\xi}$. In fact, in Proposition~\ref{prop4.3.1-1} below
we show that $\mathbf{I}^*_n$ is a quotient of $\mathbf{Bin}_n^\mathrm{ess}$.

For $\rho\in\mathbf{Bin}_n$, denote by $\overline{\rho}$ the unique 
element in $\mathbf{Bin}_n$ such that $\Gamma_{\overline{\rho}}$
is obtained from $\Gamma_\rho$ by completing each connected component of 
$\Gamma_\rho$ to a complete bipartite graph.

\begin{proposition}\label{prop4.3.1-1}
\hspace{1mm}

\begin{enumerate}[$($a$)$]
\item\label{prop4.3.1-1.1} 
Call $\rho$ and $\rho'$ in $\mathbf{Bin}_n$ equivalent provided that 
$\overline{\rho}=\overline{\rho'}$. This equivalence relation is a 
congruence on $\mathbf{Bin}_n$.
\item\label{prop4.3.1-1.2}
The subsemigroup $\mathbf{Bin}_n^\mathrm{ess}$ is a union of congruence classes.
\item\label{prop4.3.1-1.3}
The quotient of $\mathbf{Bin}_n^\mathrm{ess}$ modulo
this congruence is isomorphic to $\mathbf{I}^*_n$ via the map which
sends $\xi\in \mathbf{I}^*_n$ to the 
equivalence class of $\rho_{{}_\xi}$.
\end{enumerate}
\end{proposition}

\begin{proof}
Equivalence of $\rho$ and $\rho'$ can  be alternatively described
as follows: two vertices of $\Gamma_\rho$ belong to the same 
connected component if and only if the corresponding vertices 
of $\Gamma_{\rho'}$ belong to the same connected component. 

Let now $\rho_1$ and $\rho_2$ be two elements in $\mathbf{Bin}_n$. Then  the
product $\rho_1\rho_2$ can be described as follows in terms of 
$\Gamma_{\rho_1}$ and $\Gamma_{\rho_2}$: the graph $\Gamma_{\rho_1\rho_2}$ is
obtained by first taking the union of $\Gamma_{\rho_1}$ and $\Gamma_{\rho_2}$
under the assumption that the lower vertices of $\Gamma_{\rho_2}$ are 
identified with the corresponding upper vertices of  $\Gamma_{\rho_1}$ and then removing
these common vertices making the following adjustment of edges whenever possible
(here the {\color{red}red} vertex is removed):
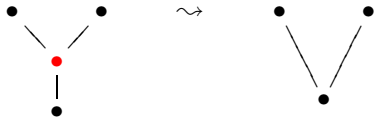
\begin{displaymath}
\xymatrix@R=3mm@C=2mm{\bullet\ar@{-}[rd]&&\bullet\ar@{-}[ld]\\
&{\color{red}\bullet}\ar@{-}[d]&\\&\bullet&}\qquad\leadsto\qquad
\xymatrix@R=3mm@C=2mm{\bullet\ar@{-}[rdd]&&\bullet\ar@{-}[ldd]\\
&&\\&\bullet&}
\end{displaymath}
Taking the union of connected graphs with a common vertex produces a
connected graph. This implies the following: assume that 
$\rho_1$ and $\rho'_1$ are equivalent and $\rho_2$ and $\rho'_2$
are equivalent. Then, for any fixed vertex, while the connected components of
this vertex in $\Gamma_{\rho_1\rho_2}$ and $\Gamma_{\rho'_1\rho'_2}$ might 
be non-isomorphic as graphs, the sets of vertices in these two components
coincide. This means exactly that $\rho_1\rho_2$ and $\rho'_1\rho'_2$
are equivalent. This shows that our equivalence relation is a congruence,
proving Claim~\eqref{prop4.3.1-1.1}.

Claim~\eqref{prop4.3.1-1.2} follows directly from the definitions.
To prove Claim~\eqref{prop4.3.1-1.3}, we note that the
map which sends $\xi\in \mathbf{I}^*_n$ to the 
equivalence class of $\rho_{{}_\xi}$ in $\mathbf{Bin}_n$ is, clearly, injective.
Moreover, it is bijective if we restrict to $\mathbf{Bin}_n^\mathrm{ess}$
modulo the congruence in Claim~\eqref{prop4.3.1-1.1}.
The fact that it is a homomorphism of semigroups follows by comparing the 
definitions of multiplications in $\mathbf{I}^*_n$ and $\mathbf{Bin}_n^\mathrm{ess}$. 
\end{proof}

\begin{remark}\label{rem4.3.1-51}
{\sf
The quotient of $\mathbf{Bin}_n$ modulo the congruence given by 
Proposition~\ref{prop4.3.1-1}\eqref{prop4.3.1-1.1} is the  inverse semigroup of
all bijections between quotients  of subsets of $\underline{n}$
(alternatively, between subsets of quotients of $\underline{n}$).
This semigroup was studied in \cite{KuMa,KMU}.
}
\end{remark}

We denote by $\boldsymbol{\pi}:\mathbf{Bin}_n^\mathrm{ess}\to \mathbf{I}^*_n$ the surjective map
given by Proposition~\ref{prop4.3.1-1}.

The monoid $\mathbf{I}^*_n$ is an inverse semigroup. The semigroup 
$\mathbf{I}^*_n$ contains a subsemigroup, denoted $\mathbf{F}^*_n$,
defined by the condition that $|B_{\sigma(i)}|=|A_i|$, for all $i$, in the 
above notation. The subsemigroup $\mathbf{F}^*_n$
is usually called the {\em maximal factorizable submonoid} of $\mathbf{I}^*_n$.
The meaning of this is the following: 
The group of units of $\mathbf{I}^*_n$ is the symmetric
group $S_n$. As usual, 
we denote by $s_i$ the elementary transposition $(i,i+1)\in S_n$, for $i=1,2,\dots,n-1$.
The idempotents of $\mathbf{I}^*_n$ are in a natural bijection with
the equivalence relations on (alternatively, the set partitions of) $\underline{n}$.
In the above notation, to get an idempotent, we take
$A_i=B_i$, for all $i$, and $\sigma=e\in S_k$. The set $E(\mathbf{I}^*_n)$
of all idempotents of $\mathbf{I}^*_n$ is a commutative semigroup whose
operation can be interpreted as ``taking the minimal equivalence relation which
contains the two given equivalence relations''. Note that $E(\mathbf{I}^*_n)$
is closed under the conjugation by elements in $S_n$. For an equivalence relation
$\rho$ on $\underline{n}$, we denote by $\xi_\rho$ the corresponding idempotent
in $\mathbf{I}^*_n$. 

The subsemigroup $\mathbf{F}^*_n$ consists of all elements of the form
$\sigma\xi$, where $\sigma\in S_n$ and $\xi\in E(\mathbf{I}^*_n)$. That 
is the {\em factorization} mentioned in the name. We note that 
$\sigma\xi=\sigma'\xi'$ does not imply $\sigma=\sigma'$, in general
(but it does imply $\xi=\xi'$). The product of two elements in 
$\mathbf{F}^*_n$ can be computed as follows:
\begin{displaymath}
(\sigma'\xi')(\sigma\xi)=(\sigma'\sigma)\big((\sigma^{-1}\xi'\sigma)\xi\big). 
\end{displaymath}
The subsemigroup $\mathbf{F}^*_n$ is an inverse semigroup as well.

For $i\neq j$ in $\{1,2,\dots,n\}$, denote by $\xi_{\{i,j\}}$ the idempotent
in $\mathbf{F}^*_n$ corresponding to the equivalence relation
with equivalence classes $\{i,j\}$ and $\{s\}$, where $s\neq i,j$.
It is easy to check that $\mathbf{F}^*_n$ is generated by $S_n$
and any $\xi_{\{i,j\}}$. For $i=1,2,\dots,n-1$, set $\xi_i:=\xi_{\{i,i+1\}}$.
Note that $\xi_{i}=\boldsymbol{\pi}(\mathbf{s}_i)$.

\subsubsection{Map from singular braid monoid}\label{s4.3.2}

\begin{proposition}\label{prop-4.3.2-1}
There is a unique homomorphism $\lambda:\mathbf{SB}(S_n)\to\mathbf{F}^*_n$ such that
\begin{displaymath}
\lambda(\sigma_{s_i})=s_i\text{ and }
\lambda(\tau_{s_i})=\xi_{i},\text{ for }i=1,2,\dots,n-1.
\end{displaymath}
This homomorphism is surjective.
\end{proposition}

\begin{proof}
The existence claim follows from the straightforward verification of the 
analogues  of the relations \eqref{sbm-eq1}, \eqref{sbm-eq2}, \eqref{sbm-eq3}, 
\eqref{sbm-eq4} and \eqref{sbm-eq5} for the elements $s_i$ and $\xi_{i}$.
The uniqueness claim follows from the fact that $\sigma_{s_i}$ and $\tau_{s_i}$
generate $\mathbf{SB}(S_n)$. Surjectivity follows by combining the fact that 
the $s_i$'th generate $S_n$ with the fact that $S_n$ and any $\xi_{i}$
generate $\mathbf{F}^*_n$.
\end{proof}

We note that $\lambda=\boldsymbol{\pi}\circ\overline{\eta}$.

\subsubsection{Presentation for $\mathbf{F}^*_n$}\label{s4.3.3}

\begin{proposition}\label{prop-4.3.3-1}
The monoid $\mathbf{F}^*_n$ is generated by the elements 
$s_i$ and $\xi_{i}$, where we have  $i=1,2,\dots,n-1$, subject to the relations
\eqref{sbm-eq1}, \eqref{sbm-eq2}, \eqref{sbm-eq3}, \eqref{sbm-eq4} and \eqref{sbm-eq5}
(for $s_i$  instead of $\sigma_{s_i}$ and $\xi_{i}$ instead  of $\tau_{s_i}$)
and,  additionally, the relations
\begin{eqnarray}
s_i^2&=&e,\quad \quad \text{ for } i=1,2,\dots,n-1; \label{prop-4.3.3-1.1}\\
\xi_{i}^2&=&\xi_{i},\quad \quad  \text{ for } i=1,2,\dots,n-1;\label{prop-4.3.3-1.2}\\
\xi_{i}\xi_{i+1}&=&\xi_{i+1}\xi_{i}, \quad 
\text{ for } i=1,2,\dots,n-2; \label{prop-4.3.3-1.3}\\
\xi_{i}s_i&=&\xi_{i},\quad \quad  \text{ for } i=1,2,\dots,n-1. \label{prop-4.3.3-1.4}
\end{eqnarray}
\end{proposition}

The associative algebra defined by our presentation was studied in \cite{Ko}
(however, the relation to the diagrammatic realization for this algebra is not
explained with all details in this paper). A slight variation of this 
presentation is given in \cite{Fi03}, see also \cite{EEF} for
a general approach to presentation of factorizable inverse monoids.
Our proof is based on a direct reduction to the presentation in \cite[Theorem~3]{Fi03}.
Our presentation can also be deduced from \cite[Theorem~6.4 and Remark 6.5]{Ea06b}.

\begin{proof}
We show that our relations imply the relations in \cite[Theorem~3]{Fi03}.
We note that, in \cite[Theorem~3]{Fi03}, the generating set for $\mathbf{F}^*_n$
is smaller than ours, namely, it consists of the Coxeter generators for $S_n$
and $\xi_{1}$. It follows from the relations \eqref{sbm-eq2} and 
\eqref{sbm-eq3} that all other $\xi_{i}$ are
conjugates of $\xi_{1}$ under the action of the symmetric group
(and hence are redundant as generators).

Relations (S) in \cite[Theorem~3]{Fi03} are just the usual Coxeter relations
for the symmetric group. Relation (F2) in \cite[Theorem~3]{Fi03} is our
relation \eqref{prop-4.3.3-1.2}, Relation (F3) in \cite[Theorem~3]{Fi03} is 
a special case of our relation \eqref{prop-4.3.3-1.4} combined with 
\eqref{sbm-eq5}, Relation (F4) in \cite[Theorem~3]{Fi03} is a special case of 
our relation \eqref{sbm-eq3}, Relation (F5) in \cite[Theorem~3]{Fi03} is a special case of 
our relation \eqref{prop-4.3.3-1.3}, and, finally, Relation (F6) in 
\cite[Theorem~3]{Fi03} is a special case of  our relation \eqref{sbm-eq4}.
Therefore our claim follows from \cite[Theorem~3]{Fi03}.
\end{proof}

\subsubsection{Type $B$ dual symmetric inverse 
monoid and its factorizable submonoid}\label{s4.3.4}

The congruence on $\mathbf{Bin}_n$ described in Proposition~\ref{prop4.3.1-1}
restricts to the monoid $\mathbf{Bin}^{B}_n$. It is natural to call
the quotient of the submonoid
\begin{displaymath}
\mathbf{Bin}^{B,\mathrm{ess}}_n:=
\mathbf{Bin}^{B}_n\cap  \mathbf{Bin}^{(2),\mathrm{ess}}_n
\end{displaymath}
of $\mathbf{Bin}^{B}_n$ modulo this restricted congruence 
the {\em type $B$ dual symmetric inverse monoid}. 
We denote it by $\mathbf{IB}^*_n$.
The maximal factorizable submonoid of $\mathbf{IB}^*_n$ is denoted
$\mathbf{FB}^*_n$.

Note that the idempotents of $\mathbf{IB}^*_n$ are the
equivalence relations on $\underline{\overline{n}}\cup\underline{n}$
which are invariant the map $i\mapsto \overline{i}$. For an equivalence class
$X$ of such an equivalence relation,  there are two possibilities:
\begin{itemize}
\item $|X\cap\{i,\overline{i}\}|\leq 1$, for all $i\in\underline{n}$,
and, in this case, $\overline{X}$ is another equivalence class;
\item $|X\cap\{i,\overline{i}\}|=2$, for some $i\in\underline{n}$,
and, in this case, $\overline{X}=X$.
\end{itemize}
Such an equivalence relation can be described by a tuple
$(\rho,Y,f)$, where
\begin{itemize}
\item $\rho$ is an equivalence relation on $\underline{n}$,
\item $Y$ is the union of some of the equivalence classes of $\rho$,
\item $f:Y\to\{\pm\}$ is a function,
\end{itemize}
modulo the equivalence relation $\sim$ defined by 
$(\rho,Y,f)\sim(\rho,Y,f')$ provided that $f$ differs from $f'$
by changing the sign on the union of some of the equivalence classes 
of $\rho$ inside $Y$. Here $Y$ consists of the absolute values of
all elements appearing in those classes 
$X\subset (\underline{\overline{n}}\cup\underline{n})$
for which $|X\cap\{i,\overline{i}\}|\leq 1$, for all $i\in\underline{n}$.
For such an $X$, the restrictions of the functions 
$f$ and $-f$ to $X$ determine $X$ and $\overline{X}$, up 
to swapping these two sets. The equivalence
relation $\sim$ compensates for this swapping.

For $i=1,2,\dots,n-1$, we denote by $\tilde{\xi}_{i}$ the idempotent of 
$\mathbf{IB}^*_n$ corresponding to the tuple $(\rho,\underline{n},f)$, where
$\rho$ is the equivalence relation on
$\underline{n}$ which has only one non-singleton part, namely $\{i,i+1\}$,
and $f(i)=+$, for all $i\in \underline{n}$. We denote by $\tilde{\xi}_{0}$ 
the idempotent of $\mathbf{IB}^*_n$ corresponding to the 
tuple $(\rho,\underline{n}\setminus\{1\},f)$, where
$\rho$ is the equality relation and $f(i)=+$, for all 
$i\in \underline{n}\setminus\{1\}$.

\subsubsection{Map from a type $B$ singular braid monoid}\label{s4.3.5}

Let $W$ be the Coxeter group corresponding to the following (type $B$) Coxeter graph:
\begin{displaymath}
\xymatrix{s_0\ar@{-}[r]^4&s_1\ar@{-}[r]&s_2\ar@{-}[r]&\dots\ar@{-}[r]&s_{n-1}} 
\end{displaymath}

The obvious analogue of the map $\lambda$ from Proposition~\ref{prop-4.3.2-1}
gives rise to a natural epimorphism from $\mathbf{SB}(W)$, where $W$ is of type $B_n$
onto $\mathbf{FB}^*_n$.

\begin{proposition}\label{prop-4.3.5-1}
There is a unique homomorphism $\boldsymbol{\lambda}:\mathbf{SB}(W)\to\mathbf{FB}^*_n$, such that
\begin{displaymath}
\boldsymbol{\lambda}(\sigma_{s_i})=\tilde{s}_i\text{ and }
\boldsymbol{\lambda}(\tau_{s_i})=\tilde{\xi}_{i},\text{ for }i=0,1,\dots,n-1.
\end{displaymath}
This homomorphism is surjective.
\end{proposition}

\begin{proof}
Existence follows from the straightforward verification of the analogues  of the relations 
\eqref{sbm-eq1}, \eqref{sbm-eq2}, \eqref{sbm-eq3}, \eqref{sbm-eq4} and \eqref{sbm-eq5} for the 
elements $\tilde{s}_i$  and $\tilde{\xi}_{i}$.  Uniqueness is due to the fact that the homomorphism
is defined on the generators of $\mathbf{SB}(W)$. Surjectivity follows from the fact that
the images of the generators of $\mathbf{SB}(W)$ generate $\mathbf{FB}^*_n$.
\end{proof}

\begin{remark}\label{rem13}
{\em 
We note that the homomorphism $\boldsymbol{\lambda}$ is the composition of 
$\boldsymbol{\eta}$ followed by the natural projection from 
$\mathbf{Bin}^{B,\mathrm{ess}}_n$ onto $\mathbf{FB}^*_n$.
}
\end{remark}

\subsubsection{Presentation for $\mathbf{FB}^*_n$}\label{s4.3.7}

\begin{proposition}\label{prop-4.3.7-1}
The monoid $\mathbf{FB}^*_n$ is generated by the elements
$\tilde{s}_i$ and $\tilde{\xi}_{i}$, where we have  $i=0,1,\dots,n-1$, subject to
\begin{itemize}
\item the (analogues of the) relations \eqref{sbm-eq1}, 
\eqref{sbm-eq2}, \eqref{sbm-eq3}, \eqref{sbm-eq4} and \eqref{sbm-eq5}; 
\item the (analogues of the) relations \eqref{prop-4.3.3-1.1}, 
\eqref{prop-4.3.3-1.2}, \eqref{prop-4.3.3-1.3} and
\eqref{prop-4.3.3-1.4}, for all $i$ including $0$;
\item the additional relations 
\begin{eqnarray}
\tilde{s}_1\tilde{\xi}_0\tilde{s}_1\tilde{\xi}_0&=&
\tilde{\xi}_0\tilde{s}_1\tilde{\xi}_0\tilde{s}_1,\label{prop-4.3.7-1-rel1}\\ 
\tilde{s}_0\tilde{\xi}_1\tilde{s}_0\tilde{\xi}_1 &=&
\tilde{\xi}_1\tilde{s}_0\tilde{\xi}_1\tilde{s}_0,\label{prop-4.3.7-1-rel2}\\ 
\tilde{\xi}_0\tilde{\xi}_1 &=&
\tilde{s}_0\tilde{\xi}_1\tilde{s}_0\tilde{\xi}_1.\label{prop-4.3.7-1-rel3}
\end{eqnarray}
\end{itemize}
\end{proposition}

\begin{proof}
Below we give a very detailed sketch of the idea of the proof,
leaving it for the reader to verify most of the technical details.

It is easy to check that the generators of $\mathbf{FB}^*_n$
satisfy all the relations in the formulation.
Denote by $Q$ the monoid with the presentation described in the 
formulation generated by $x_i$ (instead of $\tilde{s}_i$)
and $\mathtt{x}_i$ (instead of $\tilde{\xi}_i$). 
We have the obvious surjection $\psi:Q\tto \mathbf{FB}^*_n$
which we want to prove is an isomorphism. For this we need to prove
that $|Q|\leq |\mathbf{FB}^*_n|$.

Due to \eqref{sbm-eq1} and \eqref{prop-4.3.3-1.1}, the group $G$ of invertible
elements in $Q$ is isomorphic to the group of  signed permutations on $\underline{n}$
(i.e. the Coxeter group of type $B_n$).
Denote by $T$ the set of all $G$-conjugates of all $\mathtt{x}_i$,
where $i=0,1,2\dots,n-1$.

By \eqref{prop-4.3.3-1.4}, \eqref{sbm-eq2}, \eqref{sbm-eq3} and \eqref{sbm-eq5},
the $G$-conjugate stabilizer of ${\mathtt{x}}_0$ contains both 
$x_0$ and all signed permutations on $\{2,3,\dots,n\}$. Hence the number of
the $G$-conjugates of ${\mathtt{x}}_0$ is at most
$\frac{2^n\cdot n!}{2\cdot 2^{n-1}(n-1)!}=n$. Since the number of 
the $G$-conjugates of $\tilde{\xi}_0$ is exactly $n$,
we obtain that the number of
the $G$-conjugates of ${\mathtt{x}}_0$ is exactly $n$.
For $i=1,\dots,n$, we denote by $\mathtt{t}_i$ the conjugate of 
${\mathtt{x}}_0$ by the transposition $(1,i)$. Note that $\mathtt{t}_1={\mathtt{x}}_0$.

Similarly, the $G$-conjugate stabilizer of ${\mathtt{x}}_1$ contains both 
the elements $x_1$ and $x_0x_1x_0$, their product $x_1x_0x_1x_0$, and, furthermore, 
all signed permutations on $\{3,4,\dots,n\}$. Hence the number of
the $G$-conjugates of ${\mathtt{x}}_1$ is at most
$\frac{2^n\cdot n!}{4\cdot 2^{n-2}(n-2)!}=n(n-1)$. Since the number of 
the $G$-conjugates of ${\xi}_1$ is exactly $n(n-1)$,
we obtain that the number of
the $G$-conjugates of ${\mathtt{x}}_1$ is exactly $n(n-1)$.
For different $i,j\in\{1,\dots,n\}$, we denote by $\mathtt{t}_{i,j}$ 
the conjugate of ${\mathtt{x}}_1$ by $(1,i)(2,j)$.
We also denote by $\overline{\mathtt{t}}_{i,j}$ 
the conjugate of ${\mathtt{x}}_1$ by $(1,\overline{i})(2,{j})$.
From \eqref{prop-4.3.3-1.4}, it follows that $\mathtt{t}_{i,j}=\mathtt{t}_{j,i}$
and $\overline{\mathtt{t}}_{i,j}=\overline{\mathtt{t}}_{j,i}$.
Also, we have ${\mathtt{x}}_i=\mathtt{t}_{i,i+1}$, for all
$i>1$, using \eqref{sbm-eq2} and \eqref{sbm-eq3}.

From \eqref{prop-4.3.7-1-rel1}, \eqref{prop-4.3.7-1-rel2}, 
\eqref{sbm-eq3}, \eqref{sbm-eq4} and \eqref{prop-4.3.3-1.3},
it follows that all elements in $T$ commute with each other. 
From \eqref{prop-4.3.3-1.2}, it also follows that all these elements
are idempotent.

We note the following property
(involving \eqref{prop-4.3.7-1-rel1}):
\begin{equation}
\label{prop-4.3.7-1-rel1-55}
x_1\mathtt{t}_1x_1\mathtt{t}_{1,2}=\mathtt{t}_1\mathtt{t}_{1,2}.
\end{equation}
Indeed, using the already established commutativity, we have
\begin{displaymath}
\begin{array}{rcl}
x_1\mathtt{t}_1x_1\mathtt{t}_{1,2}&\overset{\eqref{prop-4.3.3-1.4}}{=}&
x_1\mathtt{t}_1\mathtt{t}_{1,2}\\
&=&
x_1\mathtt{t}_{1,2}\mathtt{t}_1\\
&\overset{\eqref{prop-4.3.3-1.4}}{=}&
\mathtt{t}_{1,2}\mathtt{t}_1\\
&=&\mathtt{t}_1\mathtt{t}_{1,2}.
\end{array}
\end{displaymath}
Also, note that $x_1\mathtt{t}_1x_1=\mathtt{t}_2$. This implies
the following extension of \eqref{prop-4.3.7-1-rel1-55}:
\begin{equation}\label{prop-4.3.7-1-rel1-55-2}
\mathtt{t}_1\mathtt{t}_{1,2}=\mathtt{t}_2\mathtt{t}_{1,2}=
\mathtt{t}_{1,2}\mathtt{t}_1=\mathtt{t}_{1,2}\mathtt{t}_2.
\end{equation}

Now we can classify the elements in the monoid 
generated by $T$. Let $\omega$ be a product of elements 
in $T$. Due to commutativity and idempotency for elements of 
$T$, we may assume  that each factor appearing in $\omega$ 
appears there only once. We associate to  $\omega$ an 
unoriented graph $\Lambda_\omega$ with the set of 
vertices $\overline{\underline{n}}\cup\underline{n}$
and the edges defined as follows:
\begin{itemize}
\item  For each factor $\mathtt{t}_i$ appearing in $\omega$,
we connect $i$ and $\overline{i}$ by an edge in $\Lambda_\omega$.
\item  For each factor $\mathtt{t}_{i,j}$, where $i\neq j$, 
appearing in $\omega$, we connect $i$ and $j$ by an edge in $\Lambda_\omega$
and we also connect $\overline{i}$ with $\overline{j}$ by an edge in $\Lambda_\omega$.
\item  For each factor $\overline{\mathtt{t}}_{i,j}$, where $i\neq j$, 
appearing in $\omega$, we connect $i$ and $\overline{j}$ by an edge in $\Lambda_\omega$
and we also connect $j$ with $\overline{i}$ by an edge in $\Lambda_\omega$.
\end{itemize}
Let $\tilde{\Lambda}_\omega$ be the graph obtained from 
$\Lambda_\omega$ by extending each connected component to a
complete graph on the set of vertices of this connected component. 

From the relations \eqref{prop-4.3.7-1-rel1}, \eqref{prop-4.3.7-1-rel2}, 
\eqref{prop-4.3.7-1-rel3} and \eqref{prop-4.3.7-1-rel1-55-2}, it follows
that $\omega=\omega'$ if and only if $\tilde{\Lambda}_\omega=
\tilde{\Lambda}_{\omega'}$. Here is an illustration of how this works by an  example.
In terms of the graphs, the relation \eqref{prop-4.3.7-1-rel3} says that 
the following two graphs define the same element of $Q$:
\begin{displaymath}
\xymatrix@C=3mm@R=3mm{\overline{1}\ar@{-}[r]\ar@{-}[d]&1\ar@{-}[d]\\\overline{2}&2}\qquad \text{ and }\qquad
\xymatrix@C=3mm@R=3mm{\overline{1}\ar@{-}[rd]\ar@{-}[d]&1\ar@{-}[d]\ar@{-}[ld]\\\overline{2}&2}\qquad 
\end{displaymath}
Using \eqref{prop-4.3.3-1.3} and \eqref{prop-4.3.3-1.4}, it then follows that 
the following two graphs define the same element of $Q$:
\begin{displaymath}
\xymatrix@C=3mm@R=3mm{\overline{1}\ar@{-}[d]&1\ar@{-}[d]\\\overline{2}\ar@{-}[r]&2}\qquad \text{ and }\qquad
\xymatrix@C=3mm@R=3mm{\overline{1}\ar@{-}[rd]\ar@{-}[d]&1\ar@{-}[d]\ar@{-}[ld]\\\overline{2}&2}\qquad 
\end{displaymath}
Now, using the idempotency of all factors, it follows that the element of $Q$
defined by any of the above graphs is equal to the element of $Q$
defined by the full graph on $\{1,2,\overline{1},\overline{2}\}$.
The general case follows from this local example inductively.

Combining the fact that $\omega=\omega'$ if and only if $\tilde{\Lambda}_\omega=
\tilde{\Lambda}_{\omega'}$ with the fact that full  symmetric
(under $i\mapsto\overline{i}$) graphs on $\overline{\underline{n}}\cup\underline{n}$
whose connected components are complete graphs classify all idempotents in 
$\mathbf{FB}^*_n$, we  obtain that the restriction of $\psi$ to the submonoid
$\langle T\rangle$ generated by $T$ in $Q$ is injective.

Finally, let $\omega$ be a product of elements in 
$T$ with the associated graph $\tilde{\Lambda}_\omega$.
Each edge of $\tilde{\Lambda}_\omega$ corresponds,
using \eqref{sbm-eq2}, \eqref{sbm-eq3}, \eqref{sbm-eq5}
and \eqref{prop-4.3.3-1.4}, to a reflection in $G$
which stabilizes $\omega$ under the left multiplication.
It is easy to check that the $\psi(G)$-stabilizer of 
$\psi(\omega)$ with respect to to the left multiplication
is generated by such reflections.
This implies that the cardinality of the stabilizer 
of $\omega$ in $G$ with respect to the left multiplication
is at least as large as that of the corresponding stabilizer 
of $\psi(\omega)$ in $\psi(G)$.
Combining the facts that both $Q$ and $\mathbf{FB}^*_n$ 
are factorizable and $\psi$ is bijective on both $G$
and $\langle T\rangle$, it follows that 
$|Q|\leq |\mathbf{FB}^*_n|$, completing the proof.
\end{proof}

\subsection{Symmetric inverse semigroup (a.k.a. the rook monoid)}\label{s4.2}

\subsubsection{Symmetric inverse semigroup}\label{s4.2.1}

For $n\in\mathbb{Z}_{\geq 1}$, consider the corresponding symmetric inverse semigroup 
$\mathbf{IS}_n$ on $\underline{n}$. The elements of $\mathbf{IS}_n$ are all bijections
between subsets of $\underline{n}$. The semigroup operation is given 
by the usual composition of partial maps,
see \cite[Section~2.1]{GM} for details. The symmetric group $S_n$ is the group of invertible
elements in $\mathbf{IS}_n$. The semigroup $\mathbf{IS}_n$ is also known as the rook monoid.

Let $\alpha\in\mathbf{IS}_n$ be a bijection from $X\subset \underline{n}$ to $Y\subset \underline{n}$.
The cardinality of $X$ is called the rank of $\alpha$, denoted $\mathrm{rank}(\alpha)$. The subset
$\widetilde{\mathbf{IS}}_n$ of $\mathbf{IS}_n$ consisting of all elements whose rank is 
different from $n-1$ is a subsemigroup (containing $S_n$). 

For a subset $X\subset\underline{n}$, the semigroup $\mathbf{IS}_n$ contains the 
identity map $\mathrm{Id}_X$ on $X$. We will use the notation $\varepsilon_X=\mathrm{Id}_{\underline{n}\setminus X}$.  
The element $\varepsilon_X$ is an idempotent and each idempotent of $\mathbf{IS}_n$ is of this form. 

The following lemma describes a close relation between $\mathbf{F}^*_n$ and $\mathbf{IS}_n$.

\begin{lemma}\label{lem-isn-dual}
Restricting an element of $\mathbf{F}^*_n$ to singletons defines a homomorphism
from $\mathbf{F}^*_n$ to $\mathbf{IS}_n$ which we denote by ${\upsilon}$.
The image of ${\upsilon}$ coincides with $\widetilde{\mathbf{IS}}_n$.
\end{lemma}

\begin{proof}
Let $\alpha$ and $\beta$ be two elements of $\mathbf{F}^*_n$. If 
$\beta$ sends some singleton $x$ to a singleton $y$ and 
$\alpha$ sends $y$ to a singleton $z$, then $\alpha\beta$ sends $x$ to $z$.

If $\beta$ sends some singleton $x$ to a singleton  $y$,
however $y$ is not a singleton for $\alpha$ but is, rather, 
contained in a non-singleton class $Y$, then 
$\alpha$ sends $Y$ to a non-singleton $Z$ (as $|\alpha(Y)|=|Y|$)
and therefore $\alpha\beta$ sends some non-singleton class containing $x$ 
to a non-singleton class containing $Y$. Similar arguments, applied to 
other possibilities of classes for $\alpha$ and $\beta$, show that the
map in the formulation is a homomorphism of monoids.

The image of this map, clearly, contains $S_n$ and $\varepsilon_{\{1,2\}}$
and these generate $\widetilde{\mathbf{IS}}_n$. Therefore
the image of our homomorphism coincides with $\widetilde{\mathbf{IS}}_n$.
\end{proof}

\subsubsection{Map from singular braid monoid}\label{s4.2.2}

\begin{proposition}\label{prop-4.2.2-1}
There is a unique homomorphism $\varphi:\mathbf{SB}(S_n)\to\mathbf{IS}_n$ such that
\begin{displaymath}
\varphi(\sigma_{s_i})=s_i\text{ and }
\varphi(\tau_{s_i})=\varepsilon_{\{i,i+1\}},\text{ for }i=1,2,\dots,n-1.
\end{displaymath}
The image of $\varphi$ coincides with $\widetilde{\mathbf{IS}}_n$.
\end{proposition}

\begin{proof}
We note that $\varphi={\upsilon}\circ \lambda$ and hence
the claim follows from Proposition~\ref{prop-4.3.2-1} and Lemma~\ref{lem-isn-dual}.
\end{proof}

The statement of Proposition~\ref{prop-4.2.2-1} admits a slight generalization when working with algebras.
For example, for a fixed commutative ring $\Bbbk$ and $a\in\Bbbk$, 
there is a unique homomorphism $\varphi_a:\mathbf{SB}(S_n)\to\Bbbk[\mathbf{IS}_n]$ such that
\begin{displaymath}
\varphi_a(\sigma_{s_i})=s_i\text{ and }
\varphi_a(\tau_{s_i})=a\varepsilon_{\{i,i+1\}},\text{ for }i=1,2,\dots,n-1.
\end{displaymath}

\subsubsection{Presentation in type $A$}\label{s4.2.3}

We have the following presentation for the monoid $\widetilde{\mathbf{IS}}_n$.

\begin{proposition}\label{prop-4.2.3-1}
The monoid $\widetilde{\mathbf{IS}}_n$ is generated by the elements 
$s_i$ and $\varepsilon_{\{i,i+1\}}$, where $i=1,2,\dots,n-1$, subject to the relations
\eqref{sbm-eq1}, \eqref{sbm-eq2}, \eqref{sbm-eq3}, \eqref{sbm-eq4} and \eqref{sbm-eq5}
(for $s_i$  instead of $\sigma_{s_i}$ and $\varepsilon_{\{i,i+1\}}$ instead  of $\tau_{s_i}$),
the relations \eqref{prop-4.3.3-1.1}, \eqref{prop-4.3.3-1.2},
\eqref{prop-4.3.3-1.3} and \eqref{prop-4.3.3-1.4}
(for $\varepsilon_{\{i,i+1\}}$ instead  of $\xi_{\{i,i+1\}}$)
and,  additionally, the relation
\begin{equation}\label{eq-newsingrelation}
\varepsilon_{\{1,2\}}\varepsilon_{\{3,4\}}=\varepsilon_{\{1,2\}}\varepsilon_{\{2,3\}}\varepsilon_{\{3,4\}}.
\end{equation}
\end{proposition}

\begin{proof}
It is clear that $s_i$ and $\varepsilon_{\{i,i+1\}}$  satisfy all these relations, so we only need to  check
that the monoid defined by this presentation is exactly $\widetilde{\mathbf{IS}}_n$.
Since $\varphi={\upsilon}\circ \lambda$, from
Proposition~\ref{prop-4.3.3-1} we only need to prove that the kernel of 
$\boldsymbol{\upsilon}$ is generated by the relation given by \eqref{eq-newsingrelation}.
Denote by $Q$ the quotient of $\mathbf{F}^*_n$ modulo the congruence generated by \eqref{eq-newsingrelation}.
Then we have  the quotient map $Q\to \widetilde{\mathbf{IS}}_n$ which is surjective by
Lemma~\ref{lem-isn-dual}.
We need to prove that $|Q|\leq  |\widetilde{\mathbf{IS}}_n|$.

If $\rho$ is an equivalence relation  on $\underline{n}$, then the 
corresponding idempotent  $\xi_\rho\in\mathbf{I}^*_n$ 
(see Subsection~\ref{s4.3.1}) is the product of
all $\xi_{\{i,j\}}$, where $i$ and $j$ belong to the same class of $\rho$. Using the
$S_n$-conjugates of the \eqref{eq-newsingrelation}, allows us to merge non-singleton classes
of equivalence relations.  In particular, we see that the images of two equivalence relations
$\rho$ and $\rho'$ in $Q$ coincide if $\rho$ and $\rho'$ have the same singletons.
Comparing the outcomes with $E(\widetilde{\mathbf{IS}}_n)$, we get that
the image of $E(\mathbf{F}^*_n)$ in $Q$ is in a natural bijection with
$E(\widetilde{\mathbf{IS}}_n)$. 

Each element of $Q$ can be written as $\sigma\xi$, where $\sigma\in S_n$
and $\xi$ is in the image of $E(\mathbf{F}^*_n)$. From the 
relation \eqref{prop-4.3.3-1.4} it follows that,
for fixed $\xi$ with $k$ singletons, the number of different elements
of the form $\sigma\xi$ is at most $\frac{n!}{(n-k)!}$. Adding this up  over
all  $\xi$, implies $|Q|\leq  |\widetilde{\mathbf{IS}}_n|$. This completes the proof.
\end{proof}

\subsubsection{Signed symmetric inverse semigroup}\label{s4.2.4}

For a positive integer $n$, consider the symmetric inverse semigroup 
$\mathbf{IS}_X$, where $X=\{-n,-n+1,\dots,-1,1,\dots,n-1,n\}$. There is the obvious 
automorphism $\overline{\omega}$ of $\mathbf{IS}_X$ induced by the 
endomorphism $\omega$ of $X$ which changes the sign. The set of all elements 
in $\mathbf{IS}_X$ that are invariant under $\overline{\omega}$ is a subsemigroup
denoted $\mathbf{SIS}_n$, the {\em  signed symmetric inverse semigroup}.

The semigroup $\mathbf{IS}_n$ can be realized as the monoid of all $n\times n$
matrices which have the property that each row and each column of the matrix
contains at most one non-zero element, and this non-zero element equals $1$.
Such a matrix can be thought of as the indicator matrix for a placement of 
rooks which do not attack each other on an $n\times n$ square board. 
This is the justification for the alias {\em rook monoid}. 

Similarly, the semigroup $\mathbf{SIS}_n$ can be realized as the monoid of all $n\times n$
matrices which have the property that each row and each column of the matrix
contains at most one non-zero element, and this non-zero element equals $\pm 1$.
In this realization, $\mathbf{SIS}_n$ is usually called the {\em signed rook monoid}
or a {\em generalized rook monoid}, see \cite{MaSr}. 

For $i=1,2,\dots,n$, we denote by $\epsilon_i\in \mathbf{SIS}_n$ the transposition
$(i,-i)$. 

\subsubsection{$\mathbf{SIS}_n$ as a quotient of $\mathbf{FB}^*_n$}\label{s4.2.45}

We have a natural type $B$ analogue of Lemma~\ref{lem-isn-dual}.
Consider the monoid $\mathbf{FB}^*_n$ and an element 
of this monoid written in the form $\sigma\xi$, where $\sigma$ is an
invertible element and $\xi$ is an idempotent. Recall from Subsection~\ref{s4.3.4}
that $\xi$ corresponds to a certain triple $(\rho,Y,f)$, up to some equivalence.
A singleton equivalence class  of $\rho$  contained in $Y$
will be called a {\em pure singleton} for $\xi$.

Note that $\sigma\xi=(\sigma\xi\sigma^{-1})\sigma$, where $\sigma\xi\sigma^{-1}$
is an idempotent. The element $\sigma$ maps each pure singleton for $\xi$
to a pure singleton for $\sigma\xi\sigma^{-1}$. In other words, each 
$\sigma\xi=(\sigma\xi\sigma^{-1})\sigma$ induces a map  from pure singletons for $\xi$
to pure singletons for $\sigma\xi\sigma^{-1}$.

\begin{lemma}\label{lem-isnB-dual}
Restricting an element of  $\mathbf{FB}^*_n$ to pure singletons defines a homomorphism
from $\mathbf{FB}^*_n$ to $\mathbf{SIS}_n$ which we denote by $\boldsymbol{\upsilon}$.
The homomorphism $\boldsymbol{\upsilon}$ is surjective.
\end{lemma}

\begin{proof}
Similar to the proof of Lemma~\ref{lem-isn-dual}.
\end{proof}

\subsubsection{Map from a type $B$ singular Artin monoid}\label{s4.2.5}

Let $W$ be the Coxeter group corresponding to the following (type $B$) Coxeter graph:
\begin{displaymath}
\xymatrix{s_0\ar@{-}[r]^4&s_1\ar@{-}[r]&s_2\ar@{-}[r]&\dots\ar@{-}[r]&s_{n-1}} 
\end{displaymath}
For $i=1,2,\dots,n-1$, we denote by $\tilde{\varepsilon}_{\{i,i+1\}}$
the idempotent of $\mathbf{SIS}_n$ given by the identity on the 
complement to $\{i,i+1,\overline{i},\overline{i+1}\}$.
Similarly, for $i=1,2,\dots,n$, we denote by $\tilde{\varepsilon}_{\{i\}}$
the idempotent of $\mathbf{SIS}_n$ given by the identity on the 
complement to $\{i,\overline{i}\}$.

\begin{proposition}\label{prop-4.2.5-1}
There is a unique homomorphism $\boldsymbol{\varphi}:\mathbf{SB}(W)\to\mathbf{SIS}_n$ such that
\begin{displaymath}
\begin{array}{rcll}
\boldsymbol{\varphi}(\sigma_{s_i})&=&\tilde{s}_i,&\text{ for }i=1,2,\dots,n-1;\\
\boldsymbol{\varphi}(\sigma_{s_0})&=&\epsilon_1&\\
\boldsymbol{\varphi}(\tau_{s_i})&=&
\tilde{\varepsilon}_{\{i,i+1\}},&\text{ for }i=1,2,\dots,n-1;\\ 
\boldsymbol{\varphi}(\tau_{s_0})&=&\tilde{\varepsilon}_{\{1\}}.&\\  
\end{array}
\end{displaymath}
The homomorphism $\boldsymbol{\varphi}$ is surjective.
\end{proposition}

Note that $\boldsymbol{\varphi}=\boldsymbol{\upsilon}\circ \boldsymbol{\lambda}$.

\begin{proof}
The first statement follows by checking the defining relations of $\mathbf{SB}(W)$
for the images of the generators under $\boldsymbol{\varphi}$ 
as prescribed in the formulation. 
The second statement follows from the easy fact that $\mathbf{SIS}_n$
is generated by $\mathbf{IS}_n$ and $\epsilon_1$.
\end{proof}

As in type $A$, there is some space for generalizations of this map in the linear
setting by using some scalars on the right hand side.

\subsubsection{Presentation in type $B$}\label{s4.2.6}

Set $\tilde{s}_0:=\epsilon_1$, $t_0:=\varepsilon_{\{1\}}$
and also $t_i:=\varepsilon_{\{i,i+1\}}$, for $i=1,2,\dots,n-1$.
We have the following presentation for ${\mathbf{SIS}}_n$.

\begin{proposition}\label{prop-4.2.6-1}
The monoid ${\mathbf{SIS}}_n$ is generated by the elements 
$\tilde{s}_i$ and $t_i$, where we have $i=0,1,\dots,n-1$, subject to
\begin{itemize}
\item the relations \eqref{sbm-eq1}, \eqref{sbm-eq2}, 
\eqref{sbm-eq3}, \eqref{sbm-eq4} and \eqref{sbm-eq5}
(for $\tilde{s}_i$  instead of $\sigma_{s_i}$ and $t_i$ instead  of $\tau_{s_i}$);
\item  the relations \eqref{prop-4.3.3-1.1}, \eqref{prop-4.3.3-1.2}, \eqref{prop-4.3.3-1.3}
and \eqref{prop-4.3.3-1.4} (including for $i=0$ and
for $\tilde{s}_i$  instead of $s_i$ and for $t_i$ instead of $\xi_i$); 
\item the relation  \eqref{eq-newsingrelation} and, additionally,
\item the relation
\begin{equation}\label{prop-4.2.6-1.1}
t_0\tilde{s}_1t_0=t_1. 
\end{equation}
\end{itemize}
\end{proposition}

\begin{proof}
It is easy to check that the generators of ${\mathbf{SIS}}_n$ satisfy all
the prescribed relations. Let $Q$ denote the semigroup generated by 
$\mathsf{s}_i$ and $\mathsf{t}_i$, for $i=0,1,\dots,n-1$, subject to the 
analogues of the relations \eqref{sbm-eq1}--\eqref{sbm-eq5}, the relations 
\eqref{prop-4.3.3-1.1}--\eqref{prop-4.3.3-1.4},
the relation  \eqref{eq-newsingrelation}
and the relation \eqref{prop-4.2.6-1.1}. 
We have a canonical surjective map $\pi:Q\to \mathbf{SIS}_n$ sending 
$\mathsf{s}_i$ to $\tilde{s}_i$ and $\mathsf{t}_i$ to $t_i$, for $i=0,1,\dots,n-1$. 
We need to show that $\pi$ is an isomorphism.

The group $W$ of units of $Q$ is generated by the elements $\mathsf{s}_i$,
where $i=0,1,\dots,n-1$, and hence is isomorphic to the group 
of signed permutations of $n$ elements (which is a Coxeter group of
type $B_n$). Consider the set 
\begin{displaymath}
T:=\{\sigma\mathsf{t}_0\sigma^{-1}\,:\,\sigma\in W\}\subset Q.
\end{displaymath}
The centralizer $C$ of $\mathsf{t}_0$ in $W$ under conjugation 
contains all $\mathsf{s}_i$,
for $i\neq 1$, due to \eqref{sbm-eq3} and \eqref{sbm-eq5}
as well as $\mathsf{s}_1\mathsf{s}_0\mathsf{s}_1$ (again,
due to \eqref{sbm-eq3}). The subgroup $N$ of
$W$ generated by all these elements has index $n$. Since
$|\pi(T)|=n$, it follows that $C=N$. Consequently, $|T|=n$ and
$\pi$ is injective, when restricted to $T$. Set
\begin{equation}\label{eq-4.2.6-2}
\mathsf{r}_1:=\mathsf{t}_0,\quad
\mathsf{r}_2:=\mathsf{s}_1\mathsf{t}_0\mathsf{s}_1,\quad
\mathsf{r}_3:=\mathsf{s}_2\mathsf{s}_1\mathsf{t}_0\mathsf{s}_1\mathsf{s}_2,
\text{ and so on}.
\end{equation}
Then $T=\{\mathsf{r}_i\,:\,i=1,2,\dots,n\}$
and $\pi(\mathsf{r}_i)=\tilde{\varepsilon}_{\{i\}}$, for $i=1,2,\dots,n$.

\begin{lemma}\label{lem-4.2.6-3}
The elements in $T$ commute.
\end{lemma}

\begin{proof}
Using conjugation by elements in $W$, it is enough to show that 
$\mathsf{r}_1$ commutes with all other $\mathsf{r}_i$.
Note that, for $i>1$, the element  $\mathsf{r}_i$ is a product 
of $\mathsf{r}_2$ and some $\mathsf{s}_j$, for $j>1$. Therefore,
using \eqref{sbm-eq3} and \eqref{eq-4.2.6-2}, it is enough to show that
$\mathsf{r}_1\mathsf{r}_2=\mathsf{r}_2\mathsf{r}_1$.

Using \eqref{prop-4.2.6-1.1} and
\eqref{prop-4.3.3-1.4}, we have
\begin{displaymath}
\mathsf{r}_1\mathsf{r}_2=\mathsf{t}_0\mathsf{s}_1\mathsf{t}_0\mathsf{s}_1
=\mathsf{t}_1\mathsf{s}_1=\mathsf{t}_1.
\end{displaymath}
Similarly, $\mathsf{r}_2\mathsf{r}_1=\mathsf{t}_1$ and we are done.
\end{proof}

Each element of $T$ is an idempotent by 
Proposition~\ref{prop-4.3.3-1}\eqref{prop-4.3.3-1.2} and now we know that all these elements commute. 
Therefore the submonoid $\langle T\rangle$ of $Q$
generated by $T$ has at most $2^n$ elements. Since $\pi(\langle T\rangle)$
has exactly $2^n$ elements, we conclude that $|\langle T\rangle|=2^n$.

Using conjugation by $W$, from $\mathsf{r}_2\mathsf{r}_1=\mathsf{t}_1$,
we obtain that all $\mathsf{t}_i$ belong to $\langle T\rangle$.
For $k=1,2,\dots,n$, using \eqref{prop-4.3.3-1.4} 
we obtain that 
\begin{displaymath}
\mathsf{q}\mathsf{r}_1\mathsf{r}_2\cdots \mathsf{r}_k=
\mathsf{r}_1\mathsf{r}_2\cdots \mathsf{r}_k
\end{displaymath}
for any $\mathsf{q}$ in the subgroup $G$ of $W$ generated
by $\mathsf{s}_0$, $\mathsf{s}_1,\dots,$ $\mathsf{s}_{k-1}$.
Note that $G$ is a Coxeter group of type $B_k$.

We can write
each element in $Q$ in the form $ab$, where $a\in W$ and $b\in \langle T\rangle$.
If $b$ is a product of $k$ different $\mathsf{r}_i$, then,
by the previous paragraph, $a$ can be chosen modulo a type $B_k$
Coxeter subgroup of $W$. Now the claim of the proposition follows
by comparing the cardinality of $Q$ with that of ${\mathbf{SIS}}_n$.
\end{proof}

The signed rook monoid and its presentation appears also in \cite{EF13}.
In the recent paper \cite{CE}, one can find an independent general
treatment of presentations for the wreath product of an abelian group
with ${\mathbf{IS}}_n$.

\subsubsection{Possible analogues in other types}\label{s4.2.7}

Propositions~\ref{prop-4.2.3-1} and \ref{prop-4.2.6-1} can be used
to extrapolate a candidate for an analogue of the symmetric inverse 
monoid for an arbitrary Coxeter group. One might note that the additional
relation given by Proposition~\ref{prop-4.2.6-1} involves an asymmetry
that is not a part of the Coxeter datum. Looking carefully into the proofs
of Propositions~\ref{prop-4.2.3-1} and \ref{prop-4.2.6-1} in order to find
some common ground, it is natural to suggest the following definition.

Let $M$ be a Coxeter matrix with the associated Coxeter group $W$, 
the Artin braid group $\mathbf{B}(W)$ and the singular Artin monoid $\mathbf{SB}(W)$.
We define the monoid $\mathbf{FB}(W)$ as the quotient of $\mathbf{SB}(W)$
modulo the relations  \eqref{prop-4.3.3-1.1}-\eqref{prop-4.3.3-1.4} in
Proposition~\ref{prop-4.3.3-1}
and, additionally, the relations that all conjugates of all $\tau_s$ by
all elements in $W$ commute. This set of relations is, clearly, not minimal
and can be cleaned up.

Let $T$ be the set of all conjugates of all $\tau_s$ by all elements in $W$.
This generates a commutative subsemigroup $\langle T\rangle$ of 
$\mathbf{FB}(W)$ consisting of idempotents. From the proofs of 
Propositions~\ref{prop-4.2.3-1} and \ref{prop-4.2.6-1}, we see that 
each element of $\mathbf{FB}(W)$ can be written as $\sigma x$,
where $\sigma\in W$ and $x\in \langle T\rangle$. In particular, this
implies that $\mathbf{FB}(W)$ is finite if and only if $W$ is.

The semigroup $\mathbf{FB}(W)$ is regular as $x \sigma^{-1}$ is
an inverse of $\sigma x$. Furthermore, it is easy to see that
$\langle T\rangle$ coincides with the set of all idempotents of $\mathbf{FB}(W)$.
In particular, all idempotents of $\mathbf{FB}(W)$ commute and hence
$\mathbf{FB}(W)$ is an inverse semigroup. Note that, in type $A$, the semigroup
$\mathbf{FB}(W)$ is isomorphic to $\mathbf{F}^*_n$
(and, in particular, is bigger than $\widetilde{\mathbf{IS}_n}$).
One could probably add the relation \eqref{eq-newsingrelation} 
for type $A$ parabolic subgroups of $W$ to get a smaller 
inverse semigroup.

In type $B$, the semigroup $\mathbf{FB}(W)$
is, in general, bigger than ${\mathbf{SIS}}_n$, as the relation
\eqref{prop-4.2.6-1.1}, being asymmetric (for example, in the dihedral types), 
cannot follow from the symmetric relations defining $\mathbf{FB}(W)$, in general.

To understand the structure and properties of $\mathbf{FB}(W)$
for general Coxeter groups seems to be a natural and interesting problem.

\subsection{Brauer algebra}\label{s4.4}

\subsubsection{The classical Brauer algebra}\label{s4.4.1}

Let $n\in \mathbb{Z}_{>0}$ and $\delta$ be an indeterminate.
The {\em classical Brauer algebra} $\mathbf{Br}_n(\mathbb{Z},\delta)$, defined in \cite{Br},
is an algebra over $\mathbb{Z}[\delta]$ with a basis given by all
{\em Brauer diagrams} on $2n$ vertices $\underline{n}\cup\underline{n}'$.
Such a Brauer diagram is a partition of $\underline{n}\cup\underline{n}'$
into two-element disjoint subsets, depicted as an unoriented graph
with the vertices from $\underline{n}$ in the top row, 
the vertices from $\underline{n}'$ in the bottom row and  edges 
between the elements belonging to the same subsets. 
Here is an example of a Brauer diagram
for $n=4$:
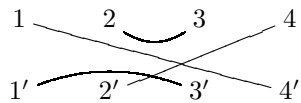
\begin{displaymath}
\xymatrix@R=5mm@C=7mm{
1\ar@{-}[rrrd]&2\ar@/_8pt/@{-}[r]&3&4\ar@{-}[dll]\\
1'\ar@/^8pt/@{-}[rr]&2'&3'&4'
}
\end{displaymath}
Multiplication is given by concatenation of two such diagrams by identifying
the primed vertices of the first with the ordinary vertices of the second,
which induces a partition of the  remaining vertices into two-element subsets,
and then multiplying with $\delta^k$, where $k$ is the number of connected  
components supported only on the identified (and subsequently removed) vertices.
Here is an example:
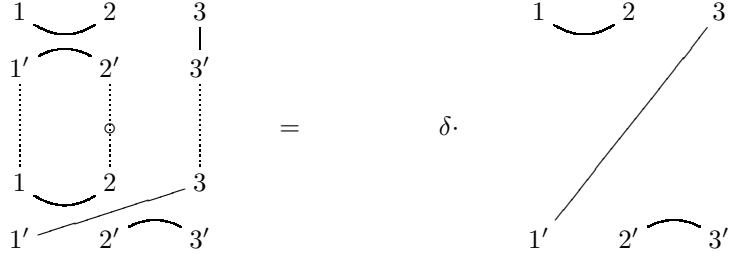
\begin{displaymath}
\xymatrix@R=3mm@C=7mm{
1\ar@/_8pt/@{-}[r]&2&3\ar@{-}[d]&&&&1\ar@/_8pt/@{-}[r]&2&3\ar@{-}[ddddll]\\
1'\ar@/^8pt/@{-}[r]\ar@{.}[dd]&2'\ar@{.}[dd]&3'\ar@{.}[dd]\\
&\circ&&=&&\delta\cdot&&\\
1\ar@/_8pt/@{-}[r]&2&3\ar@{-}[dll]\\
1'&2'\ar@/^8pt/@{-}[r]&3'&&&&1'&2'\ar@/^8pt/@{-}[r]&3'\\
}
\end{displaymath}

The diagrams in which each subset intersects both $\underline{n}$ and
$\underline{n}'$ form a group isomorphic to the symmetric group $S_n$.
In particular, $\mathbb{Z}[S_n]\subset \mathbf{Br}_n(\mathbb{Z},\delta)$.
One can extend scalars from $\mathbb{Z}$ to any commutative ring $\Bbbk$
and evaluate $\delta$ at any element $d$ of the latter ring.
This gives the algebra $\mathbf{Br}_n(\Bbbk,d)$.

\subsubsection{Map from singular braid monoid}\label{s4.4.2}

As usual, for $i=1,2,\dots,n-1$, we denote by $s_i$ the 
transposition $(i,i+1)$  considered as an element of 
$\mathbf{Br}_n(\mathbb{Z},\delta)$. We also denote by 
$\mathtt{s}_i$ the Brauer diagram
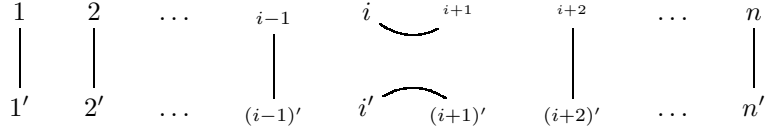
\begin{displaymath}
\xymatrix@R=7mm@C=5mm{
1\ar@{-}[d]&2\ar@{-}[d]&\dots&{}_{i-1}\ar@{-}[d]&i\ar@/_8pt/@{-}[r]&{}^{{}_{i+1}}
&{}^{{}_{i+2}}\ar@{-}[d]&\dots&n\ar@{-}[d]\\
1'&2'&\dots&{}_{(i-1)'}&i'\ar@/^8pt/@{-}[r]&{}_{(i+1)'}&{}_{(i+2)'}&\dots&n'
}
\end{displaymath}

\begin{proposition}\label{prop-4.4.2-1}
There is a unique homomorphism $\chi:\mathbf{SB}(S_n)\to
\mathbf{Br}_n(\mathbb{Z},\delta)$ such that
\begin{displaymath}
\chi(\sigma_{s_i})=s_i\text{ and }
\chi(\tau_{s_i})=\mathtt{s}_i,\text{ for }i=1,2,\dots,n-1.
\end{displaymath}
\end{proposition}

\begin{proof}
The existence claim follows from the straightforward verification of the 
analogues  of the relations \eqref{sbm-eq1}, \eqref{sbm-eq2}, \eqref{sbm-eq3}, 
\eqref{sbm-eq4} and \eqref{sbm-eq5} for the elements $s_i$ and $\mathtt{s}_i$.
The uniqueness claim follows from the fact that $\sigma_{s_i}$ and $\tau_{s_i}$
generate $\mathbf{SB}(S_n)$. 
\end{proof}

\subsubsection{Presentation for the Brauer algebra}\label{s4.4.3}

In \cite[Theorem~3.1]{KuMaz}, one can find a presentation of the Brauer
monoid $\mathbf{Br}_n$ (this corresponds to the case $\delta=1$) which is obtained from
the presentation of $\mathbf{SB}(S_n)$ by adding the additional relations
(for the images of the generators under $\chi$):
\begin{eqnarray}
s_i^2&=&e,\,\,\label{rel-br.0}\\
\mathtt{s}_i^2&=&\mathtt{s}_i,\,\, \label{rel-br.1}\\
\mathtt{s}_is_i=s_i\mathtt{s}_i&=&\mathtt{s}_i,\,\, \label{rel-br.3}\\
\mathtt{s}_i\mathtt{s}_{i\pm 1}\mathtt{s}_i&=&\mathtt{s}_i,\,\,\label{rel-br.2}\\
\mathtt{s}_i \mathtt{s}_{i\pm 1}{s}_i&=&\mathtt{s}_is_{i\pm 1},\,\,\label{rel-br.4}\\
{s}_i\mathtt{s}_{i\pm 1}\mathtt{s}_i&=&s_{i\pm 1}\mathtt{s}_i.\label{rel-br.5}
\end{eqnarray}
In other words, the kernel of $\chi$ is generated by the above relations.
We note the following:

\begin{lemma}
The relation \eqref{rel-br.2} is redundant.
\end{lemma}

\begin{proof}
We have:
\begin{displaymath}
\begin{array}{rcl}
\mathtt{s}_i\mathtt{s}_{i+1}\mathtt{s}_i
&\overset{\eqref{rel-br.0}}{=}&\mathtt{s}_i\mathtt{s}_{i+1}s_i^2\mathtt{s}_i\\
&\overset{\eqref{rel-br.3}}{=}&\mathtt{s}_i\mathtt{s}_{i+1}s_i\mathtt{s}_i\\
&\overset{\eqref{rel-br.4}}{=}&\mathtt{s}_is_{i+1}\mathtt{s}_i\\
&\overset{\eqref{rel-br.0}}{=}&s_{i+1}^2\mathtt{s}_is_{i+1}s_i^2\mathtt{s}_i\\
&\overset{\eqref{rel-br.0},\eqref{rel-br.3}}{=}&
s_{i+1}s_i^2s_{i+1}\mathtt{s}_is_{i+1}s_i\mathtt{s}_i\\
&\overset{\eqref{sbm-eq2}}{=}&
s_{i+1}s_i\mathtt{s}_{i+1}\mathtt{s}_i\\
&\overset{\eqref{rel-br.5}}{=}&
s_{i+1}s_{i+1}\mathtt{s}_i\\
&\overset{\eqref{rel-br.0}}{=}&
\mathtt{s}_i\\
\end{array}
\end{displaymath}

\end{proof}

If $\Bbbk$ is a commutative ring and $\delta\in\Bbbk$ is invertible, then
the above implies a presentation for $\mathbf{Br}_n(\Bbbk,\delta)$ by
correcting $\mathtt{s}_i^2=\mathtt{s}_i$ to $\mathtt{s}_i^2=\delta\mathtt{s}_i$.

\subsubsection{Type $B$ Brauer algebra}\label{s4.4.4}

Similarly to all previous situations (i.e. binary relations, dual symmetric inverse
semigroups and symmetric inverse semigroups), consider the Brauer
algebra $\mathbf{Br}_n^{(2)}(\mathbb{Z},\delta)$ on 
$\underline{n}\cup\underline{\overline{n}}$. It has an obvious automorphism
induced by the swapping $i\leftrightarrow\overline{i}$. We denote by 
$\mathbf{Br}_n^B(\mathbb{Z},\delta)$ the subalgebra of 
$\mathbf{Br}_n^{(2)}(\mathbb{Z},\delta)$ generated by all
Brauer diagrams on $\underline{n}\cup\underline{\overline{n}}$ that are
invariant under this automorphism. We emphasize that this is smaller than
the subalgebra of all elements in $\mathbf{Br}_n^{(2)}(\mathbb{Z},\delta)$
invariant under the above automorphism.

For $i=1,2,\dots,n-1$, we denote by $\tilde{\mathtt{s}}_i$ the unique element of
$\mathbf{Br}_n^B(\mathbb{Z},\delta)$ whose restriction to $\underline{n}$
coincides with $\mathtt{s}_i$. We also denote by $\tilde{\mathtt{s}}_0$
the element
\begin{displaymath}
\xymatrix@R=7mm@C=5mm{
\overline{n}\ar@{-}[d]&\dots&\overline{2}\ar@{-}[d]
&\overline{1}\ar@/_8pt/@{-}[r]&1&2\ar@{-}[d]&\dots&n\ar@{-}[d]\\
\overline{n}'&\dots&\overline{2}'&\overline{1}'\ar@/^8pt/@{-}[r]
&1'&2'&\dots&n'
}
\end{displaymath}
Let $\tilde{s}_0,\tilde{s}_1,\dots,\tilde{s}_{n-1}$ be the usual generators of the type $B$
Coxeter group.

A different candidate for a type $B$ Brauer algebra appeared in \cite{CL}.
As far as we can judge, this candidate does not coincide with  
our algebra. For example,  the relation \cite[2.12]{CL}, that is 
$\tilde{s}_0\tilde{s}_1\tilde{\mathtt{s}}_0=\tilde{s}_1\tilde{\mathtt{s}}_0$, does not 
hold for our generators.

\subsubsection{Map from a type $B$ singular Artin monoid}\label{s4.4.5}

Let $W$ be the Coxeter group corresponding to the following (type $B$) Coxeter graph:
\begin{displaymath}
\xymatrix{s_0\ar@{-}[r]^4&s_1\ar@{-}[r]&s_2\ar@{-}[r]&\dots\ar@{-}[r]&s_{n-1}} 
\end{displaymath}

\begin{proposition}\label{prop-4.4.5-1}
There is a unique homomorphism $\boldsymbol{\chi}:\mathbf{SB}(W)\to
\mathbf{Br}_n^B(\mathbb{Z},\delta)$ such that
\begin{displaymath}
\boldsymbol{\chi}(\sigma_{s_i})=\tilde{s}_i\text{ and }
\boldsymbol{\chi}(\tau_{s_i})=\tilde{\mathtt{s}}_i,\text{ for }i=0,1,\dots,n-1.
\end{displaymath}
\end{proposition}

\begin{proof}
The existence claim follows from the straightforward verification of the 
analogues  of the relations \eqref{sbm-eq1}, \eqref{sbm-eq2}, \eqref{sbm-eq3}, 
\eqref{sbm-eq4} and \eqref{sbm-eq5} for the elements $\tilde{s}_i$ and $\tilde{\mathtt{s}}_i$.
The uniqueness claim follows from the fact that $\sigma_{s_i}$ and $\tau_{s_i}$
generate $\mathbf{SB}(S_n)$. 
\end{proof}

\subsubsection{Connection to the partial Brauer algebra}\label{s4.4.65}

A partial analogue of Brauer monoid was introduced in \cite{Ma98}
and the representation theory of the corresponding partial
Brauer algebra was studied in \cite{MM14}. This algebra,
denoted $\mathbf{PB}_n(\delta,\delta')$, has a basis consisting of {\em partial 
Brauer diagrams}. Such a diagram is a partition of $\underline{n}\cup\underline{n}'$
into singletons and two-element disjoint subsets, depicted as an unoriented graph
with the vertices from $\underline{n}$ in the top row, 
the vertices from $\underline{n}'$ in the bottom row and  edges 
between the elements belonging to the same subsets. 
Here is an example of a partial Brauer diagram
for $n=4$:
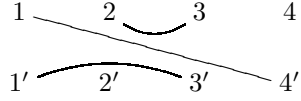
\begin{displaymath}
\xymatrix@R=5mm@C=7mm{
1\ar@{-}[rrrd]&2\ar@/_8pt/@{-}[r]&3&4\\
1'\ar@/^8pt/@{-}[rr]&2'&3'&4'
}
\end{displaymath}
Multiplication is given by concatenation of two such diagrams by identifying
the primed vertices of the first with the ordinary vertices of the second,
which induces a partition of the  remaining vertices into singletons and
two-element subsets, and then multiplying with $\delta^k(\delta')^{k'}$,
where $k$ is the number of closed connected components supported 
only on the identified (and subsequently removed) vertices
and $k'$ is the number of not closed connected components supported 
only on the identified (and subsequently removed) vertices.

Let $G$ be an abelian semigroup and $\pi$ a partial Brauer diagram. 
An {\em almost $G$-coloring} of $\pi$ is a map $f$ from the two-element
components of $\pi$ to $G$. Let $\boldsymbol{\delta}=(\delta_i)_{i\in G}$
be a collection of parameters. Given two almost $G$-colored partial
Brauer diagrams $(\pi,f)$ and $(\pi',f')$, we can define their product
as $\displaystyle (\delta')^{k'}\prod_{i\in G}\delta_i^{k_i}(\pi\pi',\mathbf{f})$, where
$\pi\pi'$ and $k'$ are as in the previous paragraph,
the value of $\mathbf{f}$ at some two-element component $C$ is defined
by taking the product of the values of $f$, respectively, $f'$ over
all two-element components of $\pi$ and $\pi'$ that contributed to $C$,
and $k_i$ denotes the number of closed connected components
of total value $i$ (computed similarly as for $C$) that are removed during the
straightening procedure.
This defined on the linear span of all almost $G$-colored partial Brauer diagrams 
the structure of an associative algebra, denoted by 
$\mathbf{APB}_n(G,\boldsymbol{\delta},\delta')$. 
It is informative to compare this  construction with the construction of 
colored partition algebras studied in \cite{MaSr22}.

By considering only the diagrams and forgetting $\boldsymbol{\delta}$ and $\delta'$, we
obtain the corresponding partial Brauer monoid $\mathbf{PB}_n$ and the
almost colored partial Brauer monoid $\mathbf{APB}_n(G)$.

\begin{proposition}\label{prop4.4.65-1}
The monoid  $\mathbf{Br}_n^B$ is isomorphic to the monoid $\mathbf{APB}_n(\mathbb{Z}_2)$.
\end{proposition}

\begin{proof}
Given a diagram $\pi\in \mathbf{Br}_n^B$, we define the corresponding diagram
$\sigma\in \mathbf{APB}_n(\mathbb{Z}_2)$ as follows (here all $i$ and $j$ are positive):
\begin{itemize}
\item If $\pi$ connects $i$ with $\overline{i}$,
then this $i$ is a singleton in $\sigma$.
\item If $\pi$ connects $i'$ with $\overline{i}'$,
then this $i'$ is a singleton in $\sigma$.
\item If $\pi$ connects $i$ with $j$,
then $i$ and $j$ are connected in $\sigma$ and colored by $0$.
\item If $\pi$ connects $i'$ with ${j}'$,
then $i'$ and $j'$ are connected in $\sigma$ and colored by $0$.
\item If $\pi$ connects $i$ with $\overline{j}$, where $i\neq j$,
then $i$ and $j$ are connected in $\sigma$ and colored by $1$.
\item If $\pi$ connects $i'$ with $\overline{j}'$, where $i\neq j$,
then $i'$ and $j'$ are connected in $\sigma$ and colored by $1$.
\item If $\pi$ connects $i$ with $j'$,
then $i$ and $j'$ are connected in $\sigma$ and colored by $0$.
\item If $\pi$ connects $i$ with $\overline{j}'$,
then $i$ and $j'$ are connected in $\sigma$ and colored by $1$.
\end{itemize}
Here is an example:

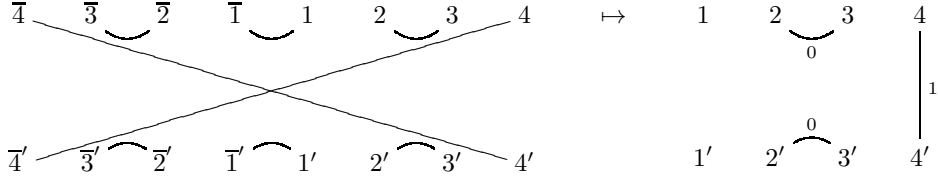
\resizebox{\textwidth}{!}{
$
\xymatrix@C=5mm@R=15mm{
\overline{4}\ar@{-}[drrrrrrr]&\overline{3}\ar@/_3mm/@{-}[r]&\overline{2}&\overline{1}
\ar@/_3mm/@{-}[r]&1&2\ar@/_3mm/@{-}[r]&3&4\ar@{-}[dlllllll]\\
\overline{4}'&\overline{3}'\ar@/^3mm/@{-}[r]&\overline{2}'&
\overline{1}'\ar@/^3mm/@{-}[r]&1'&2'\ar@/^3mm/@{-}[r]&3'&4'\\
}
\qquad\mapsto \qquad
\xymatrix@C=5mm@R=15mm{
1&2\ar@/_3mm/@{-}[r]_{0}&3&4\ar@{-}[d]^{1}\\
1'&2'\ar@/^3mm/@{-}[r]^{0}&3'&4'
}
$
}

It is easy to check that this defines an isomorphism between 
$\mathbf{Br}_n^B$ and $\mathbf{APB}_n(\mathbb{Z}_2)$. 
\end{proof}

\subsubsection{Presentation of $\mathbf{Br}_n^B(\mathbb{Z},\delta)$}\label{s4.4.6}

We now give a presentation for the type $B$ Brauer monoid $\mathbf{Br}_n^B$,
that is the semigroup formed by all type $B$ Brauer diagrams.
This corresponds to the case $\delta=1$. Note that
$\boldsymbol{\chi}$ restricts to a homomorphism from $\mathbf{SB}(S_n)$
to $\mathbf{Br}_n^B$.

\begin{theorem}\label{prop-4.4.6-1}
The monoid $\mathbf{Br}_n^B$ is generated by $\tilde{s}_i$ and 
$\tilde{\mathtt{s}}_i$, where $i=0,1,\dots,n-1$, subject to 
\begin{itemize}
\item the (analogues of the) relations \eqref{sbm-eq1}, 
\eqref{sbm-eq2}, \eqref{sbm-eq3}, \eqref{sbm-eq4} and \eqref{sbm-eq5}; 
\item the (analogues of the) relations \eqref{rel-br.0}, \eqref{rel-br.1}, 
\eqref{rel-br.3}, \eqref{rel-br.4}, \eqref{rel-br.5}, where all indices are 
greater than $0$;
\item the additional relations
\begin{eqnarray}
\tilde{\mathtt{s}}_0^2&=&\tilde{\mathtt{s}}_0,\label{rel-Bbr.0}\\
\tilde{\mathtt{s}}_1\tilde{\mathtt{s}}_0\tilde{\mathtt{s}}_1&=&\tilde{\mathtt{s}}_1,
\label{rel-Bbr.1}\\
\tilde{\mathtt{s}}_1\tilde{{s}}_0\tilde{\mathtt{s}}_1&=&\tilde{\mathtt{s}}_1,
\label{rel-Bbr.15}\\
\tilde{\mathtt{s}}_0\tilde{\mathtt{s}}_1\tilde{\mathtt{s}}_0&=&
\tilde{\mathtt{s}}_0 \tilde{s}_1\tilde{\mathtt{s}}_0,
\label{rel-Bbr.2}\\
\tilde{s}_1\tilde{\mathtt{s}}_0\tilde{s}_1\tilde{\mathtt{s}}_0=
\tilde{\mathtt{s}}_0\tilde{s}_1\tilde{\mathtt{s}}_0\tilde{s}_1&=&
\tilde{\mathtt{s}}_0 \tilde{s}_1\tilde{\mathtt{s}}_0,
\label{rel-Bbr.3}\\
\tilde{\mathtt{s}}_1\tilde{\mathtt{s}}_0&=&
\tilde{\mathtt{s}}_1 \tilde{s}_1\tilde{\mathtt{s}}_0\tilde{s}_1,
\label{rel-Bbr.4}\\
\tilde{\mathtt{s}}_0\tilde{\mathtt{s}}_1&=&
\tilde{s}_1\tilde{\mathtt{s}}_0\tilde{s}_1\tilde{\mathtt{s}}_1,
\label{rel-Bbr.5}\\
\tilde{\mathtt{s}}_0\tilde{s}_0=\tilde{s}_0\tilde{\mathtt{s}}_0&=&\tilde{\mathtt{s}}_0.
\label{rel-Bbr.6}
\end{eqnarray}
\end{itemize}
\end{theorem}

\begin{proof}
Denote by $Q$ the monoid with the presentation described in the 
formulation generated by $\tilde{x}_i$ (instead of $\tilde{s}_i$)
and $\tilde{\mathtt{x}}_i$ (instead of $\tilde{\mathtt{s}}_i$). 
We have the obvious surjection $\tilde{\boldsymbol{\pi}}:Q\tto \mathbf{Br}_n^B$
which we want to prove is an isomorphism. For this we need to prove
that $|Q|\leq |\mathbf{Br}_n^B|$.

Due to \eqref{sbm-eq1} and \eqref{rel-br.1}, the group $G$ of invertible
element in $Q$ is isomorphic to the signed permutation group on $\underline{n}$.
Denote by $T$ the set of all $G$-conjugates of all $\tilde{\mathtt{x}}_i$,
where $i=0,1,2\dots,n-1$. From \eqref{rel-Bbr.0} and \eqref{rel-br.1}
it follows that all elements in $T$ are idempotents.

By \eqref{rel-br.3}, \eqref{rel-Bbr.6}, \eqref{sbm-eq2}, \eqref{sbm-eq3} and \eqref{sbm-eq5},
the $G$-conjugate stabilizer of $\tilde{\mathtt{x}}_0$ contains both 
$\tilde{x}_0$ and all signed permutations on $\{2,3,\dots,n\}$. Hence the number of
the $G$-conjugates of $\tilde{\mathtt{x}}_0$ is at most
$\frac{2^n\cdot n!}{2\cdot 2^{n-1}(n-1)!}=n$. Since the number of 
the $G$-conjugates of $\tilde{\mathtt{s}}_0$ is exactly $n$,
we obtain that the number of
the $G$-conjugates of $\tilde{\mathtt{x}}_0$ is exactly $n$.
For $i=1,\dots,n$, we denote by $\mathtt{t}_i$ the conjugate of 
$\tilde{\mathtt{x}}_0$ by the transposition $(1,i)$.

Similarly, the $G$-conjugate stabilizer of $\tilde{\mathtt{x}}_1$ contains both 
$\tilde{x}_1$, $\tilde{x}_0\tilde{x}_1\tilde{x}_0$,
$\tilde{x}_1\tilde{x}_0\tilde{x}_1\tilde{x}_0$ and 
all signed permutations on $\{3,4,\dots,n\}$. Hence the number of
the $G$-conjugates of $\tilde{\mathtt{x}}_1$ is at most
$\frac{2^n\cdot n!}{4\cdot 2^{n-2}(n-2)!}=n(n-1)$. Since the number of 
the $G$-conjugates of $\tilde{\mathtt{s}}_1$ is exactly $n(n-1)$,
we obtain that the number of
the $G$-conjugates of $\tilde{\mathtt{x}}_1$ is exactly $n(n-1)$.
For different $i,j\in\{1,\dots,n\}$, we denote by $\mathtt{t}_{i,j}$ 
the conjugate of $\tilde{\mathtt{x}}_1$ by $(1,i)(2,j)$.
From \eqref{rel-br.3}, it follows that $\mathtt{t}_{i,j}=\mathtt{t}_{j,i}$.
Also, for different $i,j\in\{1,\dots,n\}$, we denote by 
$\overline{\mathtt{t}}_{i,j}$  the conjugate of $\tilde{\mathtt{x}}_1$ by 
$(1,\overline{i})(2,j)$. We similarly have 
$\overline{\mathtt{t}}_{i,j}=\overline{\mathtt{t}}_{j,i}$.

From \eqref{rel-Bbr.3} and the relations for
the singular Artin monoids, it follows that all $\mathtt{t}_i$ 
commute with each other. From the relations for
the singular Artin monoids, it follows also that 
$\mathtt{t}_i$ commutes with both $\overline{\mathtt{t}}_{k,j}$
and ${\mathtt{t}}_{k,j}$ provided that $i\not\in \{k,j\}$.
Similarly, both $\overline{\mathtt{t}}_{i,j}$
and $\mathtt{t}_{i,j}$ commute with both 
$\overline{\mathtt{t}}_{k,l}$ and ${\mathtt{t}}_{k,l}$
provided that $\{i,j\}\cap\{k,l\}=\varnothing$.
This implies that different elements of $T$ commute with
each other in $Q$ if and only if their images in
$\mathbf{Br}_n^B$ commute with each other.
We will call $\{i\}$ the support of $\mathtt{t}_{i}$. Similarly,
we will call $\{i,j\}$ the support of both $\overline{\mathtt{t}}_{i,j}$
and $\mathtt{t}_{i,j}$.

Let $\omega\in Q$. Denote by $\omega_l$ the product of all
$\mathtt{t}\in T$ such that $\mathtt{t}\omega=\omega$.
Define the left defect of $\omega$ as the union of 
the supports of all the factors of $\omega_l$.
Note that, if $\mathtt{t}$ is a factor of $\omega_l$,
then $\tilde{\boldsymbol{\pi}}(\mathtt{t})
\tilde{\boldsymbol{\pi}}(\omega)=\tilde{\boldsymbol{\pi}}(\omega)$.
In the monoid $\mathbf{Br}_n^B$, different 
$\tilde{\boldsymbol{\pi}}(\mathtt{t})$ with the above property
commute with each other. This implies that all 
factors of $\omega_l$ commute with each other.
In particular, the left defect of $\omega$ is the disjoint
union of the supports of the factors of $\omega_l$.
Since all elements in $T$ are idempotents, we have
$\omega=\omega_l\omega$.

Denote by $\omega_r$ the product of all
$\mathtt{t}\in T$ such that $\omega\mathtt{t}=\omega$.
Define the right defect of $\omega$ as the union of 
the supports of all the factors of $\omega_r$.
Similarly to the previous paragraph, all 
factor of $\omega_r$ commute with each other
and the right defect of $\omega$ is the disjoint
union of the supports of the factors of $\omega_r$.
Since all elements in $T$ are idempotents, we have
$\omega=\omega\omega_r$.

The following lemma will be useful.

\begin{lemma}\label{lem-n120}
Let $\omega$ be a product of pairwise commuting elements of
$T$ with disjoint supports. Assume that the left defect of
$\omega$ is $\underline{n}$. Then $\omega=\omega u\omega$,
for any $u\in Q$.
\end{lemma}

\begin{proof}
For the classical Brauer and partial Brauer algebras, 
the corresponding claim is clear. In our generalization, we note that
the presentation of the Brauer algebra is included
into our presentation by definition. The presentation of the 
partial Brauer algebra from \cite[Section~5]{KuMaz}
follows easily taking into account our additional relations 
\eqref{rel-Bbr.0}--\eqref{rel-Bbr.6}. Indeed, 
\cite[(15)]{KuMaz} follows from \eqref{rel-Bbr.0}
and the commutativity of the $\mathtt{t}_{i}$'s
established above; \cite[(16)]{KuMaz} follows from
our construction of the $\mathtt{t}_{i}$'s 
and \eqref{rel-Bbr.3}; \cite[(17)]{KuMaz} is a singular
braid relation; \cite[(18)]{KuMaz} follows from
\eqref{rel-Bbr.4} and \eqref{rel-Bbr.5} and, finally,
\cite[(19)]{KuMaz} follows from \eqref{rel-Bbr.1},
\eqref{rel-Bbr.2} and \eqref{rel-Bbr.3}.

In the terminology of Proposition~\ref{prop4.4.65-1},
we observe that we have  the similar cancellation relations \eqref{rel-Bbr.1} 
and \eqref{rel-Bbr.15} for the ``partial'' element $\tilde{\mathbf{s}}_0$ 
and the ``coloring'' element $\tilde{s}_0$, as well as the 
additional connection between these two elements 
given by the relation \eqref{rel-Bbr.6}
(this relation is, in particular, responsible for the fact that
singletons cannot be colored). Therefore cancellation of 
the factors of $u$ that are conjugates of the ``coloring'' 
element $\tilde{s}_0$ can be done similarly to the arguments
used in \cite[Section~5]{KuMaz} for the cancellation of the
conjugates of the ``partial'' element $\tilde{\mathbf{s}}_0$.
We leave the details to the reader.
\end{proof}

We claim that every $\omega\in Q$ can be written as
$\omega=\omega_l\omega_m\omega_r$, where
$\omega_l$ and $\omega_r$ are defined above
and $\omega_m$ is invertible and such that 
$\tilde{\boldsymbol{\pi}}(\omega_m)$ induces
a map from the right to the left defect of $\omega$.
We will call this a normal form of $\omega$.
We note that $\omega_l$ and $\omega_r$ are 
uniquely defined. At the same time, from
\eqref{rel-br.3} and  \eqref{rel-Bbr.6} it follows
that $\omega_m$ is not uniquely defined in general.
In fact, from (an appropriate analogue of) 
Lemma~\ref{lem-n120} it follows
that, to the very least, it is defined up to
a signed bijection between the elements in the
left and the right defects of $\omega$.

In $\mathbf{Br}_n^B$, it is easy to see that 
each element can be written
in the above normal form where the middle term
is defined exactly up to the signed 
bijection between the elements in the
left and the right defects of  $\omega$.
In particular, the sizes of the left 
and the right defects of $\omega$ coincide.
Another consequence is that $\tilde{\boldsymbol{\pi}}$ induces
a bijection between the elements of $Q$ which 
have a normal form and the elements of $\mathbf{Br}_n^B$.
In other words, to complete the proof, we just need
to show the above claim that each element in $Q$
has a normal form.

We prove the claim by induction on the number of factors
from $T$ in a shortest expression for $\omega\in Q$.
If this shortest expression of $\omega$ has no factors
from $T$, the $\omega$ is invertible and the claim is clear.
This is the basis of our induction.

To prove the induction step, consider a normal form
$\omega_l\omega_m\omega_r$, for some $\omega\in Q$. 
For any invertible $\mathtt{u}$, it is easy to see that
$(\mathtt{u}\omega_l\mathtt{u}^{-1})(\mathtt{u}\omega_m)\omega_r$ 
is a normal form for $\mathtt{u}\omega_l\omega_m\omega_r$.
Therefore, it remains to prove that, for any $\mathtt{t}\in T$, 
the element $\mathtt{t}\omega_l\omega_m\omega_r$ has a normal form.

If the support of $\mathtt{t}$ is disjoint from the left defect of
$\omega$, then, taking into account that $\mathtt{t}$ is idempotent,
we can move one copy of $\mathtt{t}$ past $\omega_l$. Note that,
because of our assumption on $\omega_m$, the support of 
$\omega_m^{-1}\mathtt{t}\omega_m\in T$ is disjoint from the
support of $\omega_r$. Using some of the commutativity relations 
that we already established, it follows that
\begin{displaymath}
(\mathtt{t}\omega_l)\omega_m(\omega_m^{-1}\mathtt{t}\omega_m\omega_r) 
\end{displaymath}
is a normal form for $\mathtt{t}\omega_l\omega_m\omega_r$.

Next, let $\mathtt{t}=\mathtt{t}_i$, for some $i$. Because of the
previous paragraph, we only need to consider the case when $i$
belongs to the left defect of $\omega$. If $\mathtt{t}_i$ is
a factor of $\omega_l$, then $\mathtt{t}\omega_l\omega_m\omega_r=\omega$
due to the idempotency of $\mathtt{t}_i$ and we have nothing to prove.
In the other case, some $\mathtt{t}_{i,j}$ or 
$\overline{\mathtt{t}}_{i,j}$ is a factor of $\omega$.
We consider the former case, with the latter being similar.
Using \eqref{rel-Bbr.5} and the the idempotency of $\mathtt{t}_i$, 
we have $\mathtt{t}_i\mathtt{t}_{i,j}=\mathtt{t}_i\mathtt{t}_j\mathtt{t}_{i,j}$.
Denote by $\omega'_l$ the element obtained from $\omega_l$ by replacing
the factor $\mathtt{t}_{i,j}$ by the factor $\mathtt{t}_i\mathtt{t}_j$.
We claim that $\omega'_l\omega_m\omega_r$ is a normal form for
$\mathtt{t}\omega_l\omega_m\omega_r$. It follows from the definitions that 
$\omega'_l\omega_m\omega_r$ is a normal form, so the only thing to show
is that $\omega'_l\omega_m\omega_r=\mathtt{t}\omega_l\omega_m\omega_r$,
that is, that $\omega'_l\omega_m\omega_r=
\omega'_l(\mathtt{t}_{i,j}\omega_m)\omega_r$.

From \eqref{rel-Bbr.1}, it follows that
\begin{equation}\label{eqn5-1-1}
\mathtt{t}_{s,t}\mathtt{t}_{s}\mathtt{t}_{s,t}=\mathtt{t}_{s,t},
\end{equation}
for all $s\neq t$. Using \eqref{rel-Bbr.2} and \eqref{rel-Bbr.3}, we have
$\tilde{\mathtt{s}}_0\tilde{\mathtt{s}}_1\tilde{\mathtt{s}}_0\tilde{\mathtt{s}}_1
\tilde{\mathtt{s}}_0\tilde{\mathtt{s}}_1\tilde{\mathtt{s}}_0
=\tilde{\mathtt{s}}_0\tilde{\mathtt{s}}_1\tilde{\mathtt{s}}_0$.
This implies
\begin{equation}\label{eqn5-1-1-2}
\mathtt{t}_{s}\mathtt{t}_{s,t}\mathtt{t}_{s}\mathtt{t}_{s,t}\mathtt{t}_{s}=
\mathtt{t}_{s}\mathtt{t}_{s,t}\mathtt{t}_{s},
\end{equation}
for all $s\neq t$. Of course, we also have similar  identities with
$\mathtt{t}_{s,t}$ replaced by $\overline{\mathtt{t}}_{s,t}$.
Using these observations, the fact that $\omega'_l\omega_m\omega_r=
\omega'_l(\mathtt{t}_{i,j}\omega_m)\omega_r$ follows from (an appropriate
analogue) of Lemma~\ref{lem-n120}.

Next consider the case $\mathtt{t}=\mathtt{t}_{i,j}$ 
with $\{i,j\}$ inside the left defect of $\omega$
(for $\mathtt{t}=\overline{\mathtt{t}}_{i,j}$, the arguments are similar).
If both, $\mathtt{t}_i$ and $\mathtt{t}_j$ are factors of 
$\omega_l$, then one can use the same argument as in the previous case.
The remaining case is when either $\mathtt{t}_{i,s}$ or
$\mathtt{t}_{j,s}$ or $\overline{\mathtt{t}}_{i,s}$
or $\overline{\mathtt{t}}_{j,s}$ is a factor of $\omega_l$,
for some $s\neq i,j$.
In this case, using \eqref{rel-br.5}, we reduce the situation to the
case of multiplication of $\omega$ with an invertible element considered above.

Finally, consider the case $\mathtt{t}=\mathtt{t}_{i,j}$ 
with only $i$ belonging to the left defect of $\omega$
(again, for $\mathtt{t}=\overline{\mathtt{t}}_{i,j}$, the arguments are similar).
The case when either $\mathtt{t}_{i,s}$ or $\overline{\mathtt{t}}_{i,s}$
is a factor of $\omega_l$ is dealt with as in the previous paragraph.
It remains to consider the case when $\mathtt{t}_i$ is a factor of
$\omega_l$. Define $\omega'_l$ as $\omega_l$ with $\mathtt{t}_i$
replaced by $\mathtt{t}_{i,j}$.

Assume that $\tilde{\boldsymbol{\pi}}(\omega_m)$ matches
$j$ with some $k$. Define $\omega'_r$ as the product 
$\omega_r\mathtt{t}_k$. We claim that $\omega'_l\omega_m\omega'_r$
is a normal form for $\mathtt{t}_{i,j}\omega$. It follows directly from
the construction that $\omega'_l\omega_m\omega'_r$
is a normal form. It remains to note that 
$\omega'_l\omega_m\omega'_r=\mathtt{t}_{i,j}\omega$
follows from the definitions of $\omega'_l$ and $\omega'_m$
and (an appropriate analogue) of Lemma~\ref{lem-n120}.
\end{proof}

If $\Bbbk$ is a commutative ring and $\delta\in\Bbbk$ is invertible, then
the above implies a presentation for $\mathbf{Br}_n^B(\Bbbk,\delta)$ by
correcting $\mathtt{s}_i^2=\mathtt{s}_i$ to $\mathtt{s}_0^2=\delta\mathtt{s}_0$
and $\mathtt{s}_i^2=\delta^2\mathtt{s}_i$, for $i>0$.

\section{Recap of category $\mathcal{O}$}\label{s5}

\subsection{Setup and definitions}\label{s5.1}

From now on we work over the field $\mathbb{C}$ of complex numbers.
We use \cite{Hu1,Hu2} as standard references.

Let $\mathfrak{g}$ be a finite dimensional, semi-simple complex Lie algebra
with a fixed triangular decomposition
\begin{equation}\label{eq6.1.1}
\mathfrak{g}=\mathfrak{n}_-\oplus \mathfrak{h}\oplus \mathfrak{n}_+. 
\end{equation}
Here $\mathfrak{h}$ is a Cartan subalgebra and $\mathfrak{n}_{\pm}$ are the positive
and the negative parts, respectively. We denote by $\mathbf{R}\subset \mathfrak{h}^*$ 
the root system of the pair $(\mathfrak{g},\mathfrak{h})$ with the induced 
decomposition $\mathbf{R}=\mathbf{R}_-\cup \mathbf{R}_+$ into negative and positive roots.
Let $W$ be the Weyl group of $\mathbf{R}$ and $\mathsf{S}$ the set of simple reflections
associated to our choice of positive roots. Then $(W,\mathsf{S})$ is a Coxeter system.
We denote by $w_0$ the longest element of $W$.

Let $\mathcal{O}$ be the Bernstein-Gelfand-Gelfand (BGG) category associated with the
triangular decomposition in \eqref{eq6.1.1}, see \cite{BGG1,BGG2}. For $\lambda\in\mathfrak{h}^*$,
we denote by
\begin{itemize}
\item $\Delta(\lambda)$ the Verma module with highest weight $\lambda$,
\item $L(\lambda)$ the unique simple quotient  of $\Delta(\lambda)$,
\item $P(\lambda)$ the indecomposable projective cover of  $L(\lambda)$ in $\mathcal{O}$,
\item $I(\lambda)$ the indecomposable injective envelope of  $L(\lambda)$ in $\mathcal{O}$,
\item $T(\lambda)$ the indecomposable tilting envelope of  $\Delta(\lambda)$ in $\mathcal{O}$.
\end{itemize}

Let $\mathcal{O}_0$ be the  principal block of $\mathcal{O}$, that is the indecomposable 
direct summand containing the trivial $\mathfrak{g}$-module $L(0)$. Let $\rho$  be the half 
of the sum of all  positive roots. The dot-action of $W$ on $\mathfrak{h}^*$ is defined
via $w\cdot\lambda:=w(\lambda+\rho)-\rho$. The simple objects in $\mathcal{O}_0$ are in bijection
with the elements in $W$ via $w\mapsto L_w:=L(w\cdot 0)$. We use similar notation for other
structural modules in $\mathcal{O}_0$, for example $\Delta_w:=\Delta(w\cdot 0)$,
$P_w:=P(w\cdot 0)$ etc.

Denote by $A$ a basic algebra of $\mathcal{O}_0$, that is the (opposite of the) 
endomorphism algebra of  a multiplicity-free projective generator of $\mathcal{O}_0$.
Then we have the usual equivalence $\mathcal{O}_0\cong A$-mod.

Denote by $\star$ the usual simple preserving duality on $\mathcal{O}$. For any
$w\in W$, we then have $L_w^\star\cong L_w$, $P_w^\star\cong I_w$ and $T_w^\star\cong T_w$.
We also have the dual Verma modules  $\nabla_w:=\Delta_w^\star$.

\subsection{Graded lift of $\mathcal{O}_0$}\label{s5.2}

Denote by $\mathtt{C}$ the coinvariant algebra of the $W$-module $\mathfrak{h}^*$.
It is defined as the quotient of the $\mathbb{Z}$-graded polynomial algebra of $\mathfrak{h}$,
under the convention $\deg \mathfrak{h}=2$, modulo the ideal generated by 
all homogeneous $W$-invariant polynomials of positive degree. Directly from the 
definition, we have that $\mathtt{C}$ is naturally $\mathbb{Z}$-graded.
In what follows, by {\em graded} we mean $\mathbb{Z}$-graded.

Soergel's combinatorial description of $\mathcal{O}_0$, see \cite{So}, asserts,
among other things, the following:
\begin{itemize}
\item $\mathrm{End}_{\mathcal{O}}(P_{w_0})\cong\mathtt{C}$;
\item the functor $\mathrm{Hom}_{\mathcal{O}}(P_{w_0},{}_-):\mathcal{O}_0\to
\mathrm{mod}\text{-}\mathtt{C}$ is full and faithful on projectives and,
for $w\in W$, the $\mathtt{C}$-module $\mathrm{Hom}_{\mathcal{O}}(P_{w_0},P_w)$
is naturally graded.
\end{itemize}
This gives the algebra $A$ the natural structure of a graded algebra
\begin{displaymath}
A=\bigoplus_{i\in\mathbb{Z}}A_i. 
\end{displaymath}
Moreover, we have $A_i=0$ for $i<0$ and $A_0$ is semisimple. In other words, 
$A$ is positively graded, see \cite{So,St} for details.

Denote by $\mathcal{O}_0^\mathbb{Z}$ the category of finite dimensional 
graded $A$-modules (where morphisms are homogeneous maps of degree zero).
There is the obvious functor from $\mathcal{O}_0^\mathbb{Z}$ to
$\mathcal{O}_0$ which forgets the grading.

We denote by $\langle 1\rangle$ the endofunctor of $\mathcal{O}_0^\mathbb{Z}$
which shifts the grading, under the normalization that it sends degree 
$0$ to degree $-1$. We fix the standard graded lifts of all structural modules:
\begin{itemize}
\item $L_w$ is concentrated in degree $0$;
\item $\Delta_w$ with the top concentrated in degree $0$;
\item $P_w$ with the top concentrated in degree $0$;
\item $\nabla_w$ with the socle concentrated in degree $0$;
\item $I_w$ with the socle concentrated in degree $0$;
\item $T_w$ such that $\Delta_w\hookrightarrow T_w$ has degree  $0$.
\end{itemize}
The duality $\star$ admits a graded lift, denoted by the same symbol.
It has the usual property that $(M^\star)_{-i}\cong \mathrm{Hom}_\mathbb{C}(M_i,\mathbb{C})$,
for $\displaystyle  M\cong \bigoplus_{i\in\mathbb{Z}}M_i\in \mathcal{O}_0^\mathbb{Z}$.

\subsection{Combinatorics of $\mathcal{O}_0$}\label{s5.3}

Consider the Grothendieck group $[\mathcal{O}_0^\mathbb{Z}]$ of $\mathcal{O}_0^\mathbb{Z}$
and the Hecke algebra $\mathbf{H}$ of the Coxeter system $(W,S)$.
There exists a unique isomorphism of abelian groups as follows:
\begin{displaymath}
\Xi:[\mathcal{O}_0^\mathbb{Z}]\to \mathbf{H}
\text{ such that }\Xi([\Delta_w])=H_w, \text{ for } w\in W, 
\text{ and }\Xi\circ \langle 1\rangle=v^{-1}\Xi.
\end{displaymath}
As a consequence  of the Kazhdan-Lusztig conjecture
(see \cite{KL}, proved in \cite{BB,BK}), we have 
$\Xi([P_w])=\underline{H}_w$,  for $w\in W$.

\subsection{Endofunctors of $\mathcal{O}_0$}\label{s5.4}

For $w\in W$, denote by $\theta_w$ the projective endofunctor on
$\mathcal{O}_0$ or $\mathcal{O}_0^\mathbb{Z}$ such that $\theta_w\ P_e=P_w$,
see \cite{BG}. These are exactly the indecomposable objects (up to shift of grading) 
in the corresponding monoidal categories $\cP$ and $\cP^\mathbb{Z}$ of projective 
endofunctors on $\mathcal{O}_0$ or $\mathcal{O}_0^\mathbb{Z}$, respectively.
The split Grothendieck ring $\left[\cP^\mathbb{Z}\right]_{\oplus}$
is isomorphic to $\mathbf{H}$ under $[\theta_w]\mapsto \underline{H}_w$, for $w\in W$,
with the usual convention that $\langle 1\rangle$ corresponds to $v^{-1}$.
Homomorphisms between projective functors are described in \cite{Bac01}.
Soergel's combinatorial description of $\mathcal{O}_0$ implies that 
$\cP(\theta_e,\theta_e)\cong\mathtt{C}$. Furthermore, each $\cP(\theta_x,\theta_y)$ is a 
free $\mathtt{C}$-module as a basis of which one can take an arbitrary lift of any basis
in $\mathrm{Hom}_{\mathcal{O}}(P_x,P_y)$. Consequently, the endomorphism algebra $B$
of $\displaystyle \bigoplus_{w\in W}\theta_w$ is positively graded.
More generally, Soergel's combinatorial description of $\mathcal{O}_0$
describes $\cP$ via the biequivalent monoidal category of
{\em Soergel bimodules} for $(W,S)$. The latter is the minimal 
isomorphism closed monoidal subcategory of $\mathtt{C}\text{-mod-}\mathtt{C}$
containing the bimodules $\mathtt{C}\otimes_{\mathtt{C}^s}\mathtt{C}$, where,
for a simple reflection $s$, we denote by $\mathtt{C}^s$ the algebra of all
$s$-invariants in $\mathtt{C}$. This biequivalence sends 
$\theta_s$ to $\mathtt{C}\otimes_{\mathtt{C}^s}\mathtt{C}$, for $s\in S$.

For $s\in S$, we  have the adjunction morphisms
$\theta_e\langle -1\rangle\overset{\mathrm{adj}_s}{\longrightarrow}\theta_s$
and $\theta_s\overset{\mathrm{adj}^s}{\longrightarrow}\theta_e\langle 1\rangle$.
The cokernel of $\mathrm{adj}_s$ is called the shuffling functor and denoted
$\mathsf{C}_s$. The kernel of $\mathrm{adj}^s$ is called the coshuffling 
functor and denoted $\mathsf{K}_s$. 

The functor $\mathsf{C}_s$ is right exact and the derived functor 
$\mathcal{L}\mathsf{C}_s$ is given by tensoring with the complex
$0\to \theta_e\langle -1\rangle\overset{\mathrm{adj}_s}{\longrightarrow}\theta_s\to 0$,
where $\theta_s$ is in the homological position $0$, followed by taking the
total complex. 

Similarly, the functor $\mathsf{K}_s$ is left exact and the derived functor 
$\mathcal{R}\mathsf{K}_s$ is given by tensoring with the complex
$0\to \theta_s\overset{\mathrm{adj}^s}{\longrightarrow}\theta_e\langle 1\rangle\to 0$,
where $\theta_s$ is in the homological position $0$, followed by taking the
total complex. The functors $\mathcal{L}\mathsf{C}_s$ and $\mathcal{R}\mathsf{K}_s$
are mutually inverse equivalences. We refer to \cite{Ca,MaStr} for details. 

For $w\in W$ and $s\in S$ such that $ws>w$, we have the exact sequences
\begin{displaymath}
\Delta_{w}\langle -1\rangle
\overset{\mathrm{adj}_s(\Delta_w)}{\hookrightarrow} \theta_s \Delta_{w}\tto \Delta_{ws}\quad
\text{ and }\quad
\Delta_{w}\hookrightarrow \theta_s \Delta_{ws}\overset{\mathrm{adj}^s(\Delta_{ws})}{\tto} 
\Delta_{ws}\langle 1\rangle.
\end{displaymath}
In particular,  we have
\begin{equation}\label{eq5.4.0}
\mathsf{C}_s \Delta_{w}\cong \Delta_{ws}\quad
\text{ and }\quad
\mathsf{K}_s \Delta_{ws}\cong \Delta_{w}.
\end{equation}

\begin{corollary}\label{cor5.4.1}
Let $w\in W$ and $w=s_1s_2\dots s_k$ be a reduced expression. 
Then we have an isomorphism
$\mathsf{C}_{s_k}\dots \mathsf{C}_{s_2}\mathsf{C}_{s_1}\Delta_e\cong \Delta_w$.
\end{corollary}

\begin{proof}
Apply \eqref{eq5.4.0} inductively along the reduced expression.
\end{proof}

\subsection{Homotopy category of complexes of projective functors}\label{s5.5}

Consider the homotopy category $\mathcal{K}^b(\cP^\mathbb{Z})$ of complexes of
projective functors. This category has the natural structure of a tensor 
triangulated category since it is the homotopy  category of an additive 
finitary monoidal category. The category $\mathcal{K}^b(\cP^\mathbb{Z})$ acts naturally 
on $\mathcal{D}^b(\mathcal{O}_0^\mathbb{Z})$ via the component-wise action 
followed by taking the total complex. The following elementary observation will 
be very handy:

\begin{lemma}\label{lem5.5.1}
Let $\mathsf{F}_\bullet\in \mathcal{K}^b(\cP^\mathbb{Z})$ be minimal (in the sense that
it does not have non-trivial null-homotopic direct summands). Then
$\mathsf{F}_\bullet\Delta_e\in \mathcal{K}^b(\mathrm{proj}(\mathcal{O}_0^\mathbb{Z}))$
is also minimal.
\end{lemma}

\begin{proof}
A complex of projective modules is not minimal provided that it contains a
(null-homotopic) direct summand of the form $0\to P\overset{\cong}{\longrightarrow} P\to 0$,
for some indecomposable $P$.

Let $\theta_w\langle i\rangle\overset{\alpha}{\longrightarrow} \theta_w\langle i\rangle$
be a piece of $\mathsf{F}_\bullet$ which, after the evaluation at $\Delta_e$, produces
a  null-homotopic direct summand 
$P_w\langle i\rangle  \overset{\alpha(\Delta_e),\cong}{\longrightarrow} P_w\langle i\rangle$
of $\mathsf{F}_\bullet\Delta_e$. 
Then $\alpha$ must be non-zero and thus an isomorphism due to the positivity of 
the grading on $B$. This contradicts our assumption that $\mathsf{F}_\bullet$ is minimal.
\end{proof}

\section{Original categorification of singular braid monoid using 
projective functors}\label{s6}

\subsection{Braid group action}\label{s6.1}

Let $G$ be a (semi)group and $\mathcal{C}$ a category. Recall that a {\em weak action}
of $G$  on $\mathcal{C}$ is an assignment of an endofunctor $F_g:\mathcal{C}\to\mathcal{C}$, 
to each $g\in G$, such that $F_gF_h\cong F_{gh}$, for any $g,h\in G$. The following
statement is well-known (see \cite{MaStr,MaStr2}). However, there does not seem to be an available
proof in the literature in full generality, so we give a proof that appears to be based on a new
approach that we will also extensively use later on.

\begin{proposition}\label{prop6.1.1}
The assignment $\sigma_s\mapsto \mathcal{L}\mathsf{C}_s$ and
$\sigma_s^{-1}\mapsto \mathcal{R}\mathsf{K}_s$, for $s\in S$, extends uniquely
(up to isomorphism) to a weak action of the Artin group $\mathbf{B}(W)$ on 
$\mathcal{D}^b(\mathcal{O}_0^\mathbb{Z})$.
\end{proposition}

\begin{proof}
We need to check  that every defining relation for  $\mathbf{B}(W)$ gives rise to an isomorphism
of functors between the two sides of the relation. The relations
$\sigma_s\sigma_s^{-1}=\sigma_s^{-1}\sigma_s=e$ are checked in \cite[Theorem~5.9]{MaStr}, 
this means exactly that the functors $\mathcal{L}\mathsf{C}_s$ and $\mathcal{R}\mathsf{K}_s$
are mutually inverse equivalences.
Hence we only need to check the braid relations.

Let $s\neq t\in S$ and $w=\underbrace{sts\dots}_{m_{s,t}\text{ factors}}$.
Consider the functor
\begin{displaymath}
\mathsf{F}=\underbrace{ \dots 
\mathcal{L}\mathsf{C}_s\mathcal{L}\mathsf{C}_t\mathcal{L}\mathsf{C}_s
}_{m_{s,t}\text{ factors}}.
\end{displaymath}
Evaluating $\mathsf{F}$ at $\Delta_e$ and noting that all shuffling functors are
acyclic on Verma modules, we obtain $\mathsf{F}\Delta_e\cong \Delta_w$ using Corollary~\ref{cor5.4.1}.

Let $\mathsf{G}_\bullet\in \mathcal{K}^b(\cP^\mathbb{Z})$ be a  minimal complex representing 
$\mathsf{F}$. Then, by Lemma~\ref{lem5.5.1}, the evaluation $\mathsf{G}_\bullet\Delta_e$
gives a minimal projective resolution of $\Delta_w$. To complete the proof, we need to show
that this property describes $\mathsf{G}_\bullet$ uniquely, up to isomorphism.

For $i>0$, denote by $\mathbf{s}_i$ the word $stst\dots$ of length $i$
and by $\mathbf{t}_i$ the word $tsts\dots$ of length $i$.
Since the parabolic subgroup of $W$ generated by $s$ and $t$ has rank two,
all KL-polynomials for this parabolic subgroup
are trivial and hence the minimal projective resolution of $\Delta_w$
is easy to describe. The Bruhat interval $[e,w]$ in $W$ has the following form:
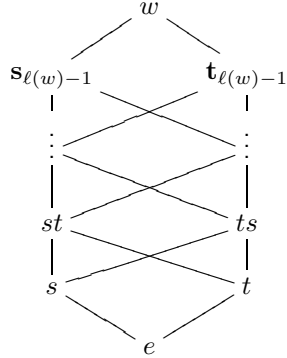
\begin{displaymath}
\xymatrix@C=4mm@R=4mm{
&w\ar@{-}[dr]\ar@{-}[dl]&\\
\mathbf{s}_{\ell(w)-1}\ar@{-}[d]\ar@{-}[drr]&&\mathbf{t}_{\ell(w)-1}\ar@{-}[d]\ar@{-}[dll]\\
\vdots\ar@{-}[d]\ar@{-}[drr]&&\vdots\ar@{-}[d]\ar@{-}[dll]\\
st\ar@{-}[d]\ar@{-}[drr]&&ts\ar@{-}[d]\ar@{-}[dll]\\
s\ar@{-}[dr]&&t\ar@{-}[dl]\\
&e&\\
}
\end{displaymath}
A minimal projective resolution $\mathcal{P}^\bullet_w$ of $\Delta_w$ then has the following form:

\resizebox{\textwidth}{!}{
$
0\to P_e\langle-\ell(w)\rangle \to
P_s\langle-\ell(w)+1\rangle\oplus P_t\langle-\ell(w)+1\rangle \to\dots 
\to
P_{\mathbf{s}_{\ell(w)-1}}\langle-1\rangle
\oplus P_{\mathbf{t}_{\ell(w)-1}}\langle-1\rangle 
\to P_w\to 0. 
$
}

For any two summands $P_x\langle -i\rangle$ and $P_y\langle -i+1\rangle$   in the 
neighboring positions of this resolution, there is a unique, up to scalar, map
from $P_x\langle -i\rangle$ to $P_y\langle -i+1\rangle$. Since,  in this case, $x\neq y$,
there are no non-zero homomorphisms from $P_x\langle -i\rangle$ to $P_y\langle -i\rangle$. 
Using Backelin's description of homomorphisms between projective functors, 
this implies that there is a unique, up to scalar, map
from $\theta_x\langle -i\rangle$ to $\theta_y\langle -i+1\rangle$ as well.
Similarly, for any two summands $P_x\langle -i\rangle$ and $P_y\langle -i+2\rangle$ 
in the positions of this resolution with one gap, there is a unique, up to scalar, map
from $P_x\langle -i\rangle$ to $P_y\langle -i+2\rangle$ and also 
from $\theta_x\langle -i\rangle$ to $\theta_y\langle -i+2\rangle$.

Finally, note that the evaluation map gives rise to a morphism of complexes
$\mathsf{G}_\bullet\to \mathcal{P}^\bullet_w$. If we now fix $\mathcal{P}^\bullet_w$
and look at $\mathsf{G}_\bullet\to \mathcal{P}^\bullet_w$, the properties described in previous 
paragraph identify $\mathsf{G}_\bullet$ uniquely, up to isomorphism. This completes the proof.
\end{proof}

\begin{remark}\label{rem6.1.4}
{\rm
It is interesting to see how the argument used in the above proof relates to
na{\"\i}ve computations of compositions of derived shuffling functors.
For $w=sts$, the na{\"\i}ve composition followed by taking the total complex
evaluates the composition $\mathcal{L}\mathsf{C}_s\mathcal{L}\mathsf{C}_t\mathcal{L}\mathsf{C}_s$ to 

\resizebox{\textwidth}{!}{
$
0\to \theta_e\langle -3\rangle\to
\theta_s\langle -2\rangle\oplus \theta_t\langle -2\rangle\oplus\theta_s\langle -2\rangle
\to
\theta_s\theta_t\langle -1\rangle\oplus 
\theta_s\theta_s\langle -1\rangle\oplus\theta_t\theta_s\langle -1\rangle
\to
\theta_s\theta_t\theta_s\to 0
$
}

Further, the computation of all individual compositions gives:

\resizebox{\textwidth}{!}{
$
0\to \theta_e\langle -3\rangle\to
\theta_s\langle -2\rangle\oplus \theta_t\langle -2\rangle\oplus{\color{violet}\theta_s\langle -2\rangle}
\to
\theta_{ts}\langle -1\rangle\oplus 
{\color{magenta}\theta_s\langle -2\rangle\oplus\theta_s} \oplus\theta_{st}\langle -1\rangle
\to
\theta_{sts}\oplus {\color{violet}\theta_s}\to 0
$
}

Since the projective resolution of $\Delta_w$ is linear (see, e.g., \cite[Corollary~2.8]{Ma}), 
the indecomposables appearing 
at homological position $i$ of a minimal representative of the above complex must be shifted 
by $i$. In other words, the summands ${\color{magenta}\theta_s\langle -2\rangle\oplus\theta_s}$
must belong to null-homotopic summands and should be removed. This removal forces the removal of
the summands ${\color{violet}\theta_s}$ and ${\color{violet}\theta_s\langle -2\rangle}$ resulting
in the minimal complex
\begin{displaymath}
0\to \theta_e\langle -3\rangle\to
\theta_s\langle -2\rangle\oplus \theta_t\langle -2\rangle
\to
\theta_{ts}\langle -1\rangle\oplus 
\theta_{st}\langle -1\rangle
\to
\theta_{sts}\to 0
\end{displaymath}
which is exactly the one that appears in the proof above.
} 
\end{remark}

\subsection{Singular crosses via projective functors}\label{s6.2}

In type $A$, the following result can be found in \cite[Theorem~1]{MaStr2} and
\cite[Theorem~5.1]{CM}. We give a general proof.

\begin{theorem}\label{thm6.2.1}
The assignment $\sigma_s\mapsto \mathcal{L}\mathsf{C}_s$ and
$\tau_s\mapsto \theta_s$, for $s\in S$, extends uniquely
(up to isomorphism) to a weak action of the singular Artin monoid 
$\mathbf{SB}(W)$ on $\mathcal{D}^b(\mathcal{O}_0^\mathbb{Z})$.
\end{theorem}

\begin{proof}
After Proposition~\ref{prop6.1.1}, we only need to check the relations involving singular 
generators. That, under the assumption  $st=ts$, we have 
$\theta_s\theta_t\cong \theta_t\theta_s\cong\theta_{st}$ follows directly from the
classification of projective functors.

Now let us check the braid relations for $m_{s,t}>2$, so we assume that.
We do the case of even $m_{s,t}$ and the odd case is similar.
Let $w=\mathbf{t}_{m_{s,t}-1}$ in the notation of the proof of Proposition~\ref{prop6.1.1}.
Then our braid relation reads  $sw=ws$. We will use the same approach as in the proof of
Proposition~\ref{prop6.1.1}. We show that the functors on the two sides of this relation send
$\Delta_e$ to the same module, up to isomorphism, and that the projective resolution  of 
this module has similar properties as the projective resolutions of Verma modules used
in the proof of Proposition~\ref{prop6.1.1}.

Consider the functors 
\begin{displaymath}
\mathsf{F}=\theta_s\circ\mathcal{L}\mathsf{C}_w\quad\text{ and }\quad
\mathsf{G}=\mathcal{L}\mathsf{C}_w\circ\theta_s.
\end{displaymath}

\begin{lemma}\label{lem6.2.2}
We have $\mathsf{F}\Delta_e\cong \mathsf{G}\Delta_e\cong\theta_s\Delta_w$. 
\end{lemma}

\begin{proof}
That $\mathsf{F}\Delta_e\cong\theta_s\Delta_w$ follows directly from 
Subsection~\ref{s5.4}. Also from Subsection~\ref{s5.4} it follows that 
there is a short exact sequence 
\begin{equation}\label{eq6.2.3}
\Delta_w\langle -1\rangle\hookrightarrow \mathsf{G}\Delta_e\tto \Delta_{ws}.
\end{equation}
The module $\theta_s\Delta_e\cong P_s$ is indecomposable. Since
$\mathcal{L}\mathsf{C}_w$ is an equivalence, we have that 
$\mathsf{G}\Delta_e$ is indecomposable as well.
Therefore the sequence in \eqref{eq6.2.3} is nonsplit.
Now, the fact that $\mathsf{F}\Delta_e\cong \mathsf{G}\Delta_e$
follows from $\dim\mathrm{Ext}^1_{\mathcal{O}}(\Delta_{ws},\Delta_w\langle -1\rangle)=1$
(see, e.g., \cite[Corollary~23]{Ma07}).
\end{proof}

\begin{lemma}\label{lem6.2.3}
The minimal  projective resolution of $\theta_s\Delta_w$ has the following form:
\begin{displaymath}
0\to P_{tws}\langle -1\rangle\to P_{ws}\to  0. 
\end{displaymath}
\end{lemma}

\begin{proof}
We start by noting  that $\ell(tws)=\ell(ws)-1$. There is a unique, up to scalar, map 
from $P_{tws}\langle -1\rangle$ to $P_{ws}$. 

Note that both $w$ and $tws$ belong
to a rank two parabolic subgroup of $W$. Since the Kazhdan-Lusztig theory in the
rank two case is trivial, the multiplicities of Verma modules in 
$P_{tws}\langle -1\rangle$ and $P_{ws}$ are either $0$ or $1$ and are determined
by the Bruhat order. This implies that  both $P_{tws}\langle -1\rangle$ and $P_{ws}$
have simple socles (by \cite[Theorem~8.1]{St}) and, consequently, any non-zero map 
from $P_{tws}\langle -1\rangle$ to $P_{ws}$ is injective.  

By comparing the Verma flags of $P_{tws}\langle -1\rangle$ and $P_{ws}$,
we see that the images of $\theta_s\Delta_w$ and $P_{ws}/P_{tws}\langle -1\rangle$
in the Grothendieck group coincide. The module $P_{ws}/P_{tws}\langle -1\rangle$
has, by construction,  simple top. In particular, it is indecomposable.
Now, that fact that we have an isomorphism $P_{ws}/P_{tws}\langle -1\rangle\cong \theta_s\Delta_w$
follows from
$$\dim\mathrm{ext}^1_{\mathcal{O}}(\Delta_{ws},\Delta_w\langle -1\rangle)=1,$$
see, e.g., \cite[Corollary~23]{Ma07}.
\end{proof}

Taking Lemmata~\ref{lem6.2.2} and \ref{lem6.2.3} into account,
the same reasoning as in the proof of Proposition~\ref{prop6.1.1}
shows that both $\mathsf{F}$ and $\mathsf{G}$ have the following minimal
representative (where the middle map is unique up to scalar and non-zero)
which is unique, up to isomorphism:
\begin{equation}\label{eqaaa111}
0\to \theta_{tws}\langle -1\rangle\to \theta_{ws}\to  0. 
\end{equation}

It remains to check the relation $\tau_s\sigma_s=\sigma_s\tau_s$.
We have $$\theta_s\circ\mathcal{L}\mathsf{C}_s\,\Delta_e\cong
\mathcal{L}\mathsf{C}_s\circ\theta_s\,\Delta_e\cong P_s\langle 1\rangle.$$
The argument from the proof of Proposition~\ref{prop6.1.1}
implies $\theta_s\circ\mathcal{L}\mathsf{C}_s
\cong\mathcal{L}\mathsf{C}_s\circ\theta_s$.
This completes the proof.
\end{proof}

\section{Categorification of the desingularization map}\label{s7}

\subsection{New candidate for the singular crossing}\label{s7.1}

For a simple reflection $s$, consider the following morphism  $\xi$ in 
$\mathcal{K}^b(\cP^\mathbb{Z})$, where $\theta_s$ is 
in the homological position $0$:
\begin{displaymath}
\xymatrix{
0\ar[r]\ar[d]&0\ar[r]\ar[d]&\theta_e\langle 1\rangle\ar[r]^{\mathrm{adj}_s}\ar@{=}[d]
&\theta_s\langle 2\rangle\ar[r]\ar[d]&0\ar[d]\\
0\ar[r]&\theta_s\ar[r]^{\mathrm{adj}^s}&\theta_e\langle 1\rangle\ar[r]&0\ar[r]&0
} 
\end{displaymath}
Setting $\alpha_s=\mathrm{adj}_s\mathrm{adj}^s$, we obtain that the complex
\begin{equation}\label{eq7.1.1}
0\longrightarrow \theta_s\overset{\alpha_s}{\longrightarrow} 
\theta_s\langle 2\rangle \longrightarrow 0
\end{equation}
is a minimal representative for the cone of $\xi$. Note that the identity map on
$\theta_e\langle 1\rangle$ splits off as a direct summand of the cone and
hence disappears in this minimal representative. We denote 
the complex in \eqref{eq7.1.1} by $\hat{\theta}_s$.

We note that $\alpha_s$ is non-zero since, when evaluated at the simple
module $L_s$, the adjunction morphism $\mathrm{adj}^s$ is surjective
while the adjunction morphism $\mathrm{adj}_s$ is injective.

\subsection{Main theorem}\label{s7.2}

\begin{theorem}\label{thm7.2.1}
The assignment $\sigma_s\mapsto \mathcal{L}\mathsf{C}_s$ and
$\tau_s\mapsto \hat{\theta}_s$, for $s\in S$, extends uniquely
(up to isomorphism) to a weak action of the singular Artin monoid 
$\mathbf{SB}(W)$ on $\mathcal{D}^b(\mathcal{O}_0^\mathbb{Z})$.
\end{theorem}

Since $\hat{\theta}_s$ is defined as a cone of a morphism from 
$\mathcal{L}\mathsf{C}_s[-2]\langle 2\rangle$
to $\mathcal{L}\mathsf{C}_s$, one can interpret
Theorem~\ref{thm7.2.1} as a categorification of the desingularization map.

Note that Relation~\ref{sbm-eq1} is proved in Theorem~\ref{prop6.1.1},
so, to prove Theorem~\ref{thm7.2.1}, we only need to check
Relations~\ref{sbm-eq2}, \ref{sbm-eq3}, \ref{sbm-eq4} and \ref{sbm-eq5}.
We will do this, one at a time, below.

\subsection{Proof: Relation~\ref{sbm-eq4}}\label{s7.3}

We need to prove that, for $s,t\in S$, we have 
\begin{equation}\label{eq7.3.1-1}
\hat{\theta}_s\circ \hat{\theta}_t
\cong \hat{\theta}_t\circ \hat{\theta}_s
\end{equation}
provided that $st=ts$. 
The LHS of \eqref{eq7.3.1-1}
is given by taking the total complex of the following commutative diagram:
\begin{equation}\label{eq7.3.1-3}
\xymatrix{
\theta_s\theta_t\ar[rr]^{\mathrm{id}_s\circ_h\alpha_t} 
\ar[d]_{\alpha_s\circ_h\mathrm{id}_t}&& 
\theta_s\theta_t\langle 2\rangle\ar[d]^{\alpha_s\circ_h\mathrm{id}_t\langle 2\rangle}\\
\theta_s\theta_t\langle 2\rangle
\ar[rr]^{\mathrm{id}_s\circ_h\alpha_t\langle 2\rangle} 
&& \theta_s\theta_t\langle 4\rangle\\
} 
\end{equation}
Using Soergel's combinatorial description, we can represent this
diagram as the following complexes of bimodules over the coinvariant 
algebra  $\mathtt{C}$:
\begin{equation}\label{eq7.3.1-5}
\xymatrix{
\mathtt{C}\otimes_{\mathtt{C}^s}\mathtt{C}
\otimes_{\mathtt{C}^t}\mathtt{C}
\ar[rr]^{\mathrm{id}_s\circ_h\alpha_t} 
\ar[d]_{\alpha_s\circ_h\mathrm{id}_t}&& 
\mathtt{C}\otimes_{\mathtt{C}^s}\mathtt{C}
\otimes_{\mathtt{C}^t}\mathtt{C}
\langle 2\rangle\ar[d]^{\alpha_s\circ_h\mathrm{id}_t\langle 2\rangle}\\
\mathtt{C}\otimes_{\mathtt{C}^s}\mathtt{C}
\otimes_{\mathtt{C}^t}\mathtt{C}\langle 2\rangle
\ar[rr]^{\mathrm{id}_s\circ_h\alpha_t\langle 2\rangle} 
&&\mathtt{C}\otimes_{\mathtt{C}^s}\mathtt{C}
\otimes_{\mathtt{C}^t}\mathtt{C}\langle 4\rangle\\
}  
\end{equation}
here $\mathtt{C}^s$ and $\mathtt{C}^t$ are the subalgebras of 
$s$- and $t$-invariants in $\mathtt{C}$, respectively.
Let $x_s$ and $x_t$ denote the images in $\mathtt{C}$ of the 
coroots corresponding to $s$ and $t$, respectively. Then the 
map $\alpha_s$ maps the generator $1\otimes 1$ of 
$\mathtt{C}\otimes_{\mathtt{C}^s}\mathtt{C}$ to
$1\otimes x_s+ x_s\otimes 1$. Similarly, the 
map $\alpha_t$ maps the generator $1\otimes 1$ of 
$\mathtt{C}\otimes_{\mathtt{C}^t}\mathtt{C}$ to
$1\otimes x_t+ x_t\otimes 1$. 
  
Since $st=ts$, we have the unique isomorphism from 
the indecomposable $\mathtt{C}$-$\mathtt{C}$-bimodule
$\mathtt{C}\otimes_{\mathtt{C}^s}\mathtt{C}
\otimes_{\mathtt{C}^t}\mathtt{C}$ to
the indecomposable $\mathtt{C}$-$\mathtt{C}$-bimodule
$\mathtt{C}\otimes_{\mathtt{C}^t}\mathtt{C}
\otimes_{\mathtt{C}^s}\mathtt{C}$ which sends 
$1\otimes 1\otimes 1$ to $1\otimes 1\otimes 1$.
Applying this ismorphism component-wise to \eqref{eq7.3.1-5}
and taking the formulae of the previous paragraph into account,
we get 
\begin{equation}\label{eq7.3.1-55}
\xymatrix{
\mathtt{C}\otimes_{\mathtt{C}^t}\mathtt{C}
\otimes_{\mathtt{C}^s}\mathtt{C}
\ar[rr]^{\alpha_t\circ_h\mathrm{id}_s} 
\ar[d]_{\mathrm{id}_t\circ_h\alpha_s}&& 
\mathtt{C}\otimes_{\mathtt{C}^t}\mathtt{C}
\otimes_{\mathtt{C}^s}\mathtt{C}
\langle 2\rangle\ar[d]^{\mathrm{id}_t\circ_h\alpha_s\langle 2\rangle}\\
\mathtt{C}\otimes_{\mathtt{C}^t}\mathtt{C}
\otimes_{\mathtt{C}^s}\mathtt{C}\langle 2\rangle
\ar[rr]^{\alpha_t\circ_h\mathrm{id}_s\langle 2\rangle} 
&&\mathtt{C}\otimes_{\mathtt{C}^t}\mathtt{C}
\otimes_{\mathtt{C}^s}\mathtt{C}\langle 4\rangle\\
}  
\end{equation}
which represents the RHS of \eqref{eq7.3.1-1}
(note that the two middle bimodules also got swapped compared to 
\eqref{eq7.3.1-5}).

\subsection{Proof: Relation~\ref{sbm-eq5}}\label{s7.4}

We need to prove that 
\begin{equation}\label{eq7.4.1-1}
\hat{\theta}_s\circ \mathcal{L}\mathsf{C}_s
\cong \mathcal{L}\mathsf{C}_s\circ \hat{\theta}_s. 
\end{equation}
The LHS of \eqref{eq7.4.1-1}
is given by taking the total complex of the following commutative diagram:
\begin{equation}\label{eq7.4.1-3}
\xymatrix{
\theta_s\langle -1\rangle\ar[rr]^{\mathrm{id}_s\circ_h\mathrm{adj}_s} 
\ar[d]_{\alpha_s\langle -1\rangle}&& 
\theta_s\theta_s\ar[d]^{\alpha_s\circ_h\mathrm{id}_s}\\
\theta_s\langle 1\rangle
\ar[rr]^{\mathrm{id}_s\circ_h\mathrm{adj}_s\langle 2\rangle} 
&& \theta_s\theta_s\langle 2\rangle\\
} 
\end{equation}
Using Soergel's combinatorial description, we can look at the image of
\eqref{eq7.4.1-3} in the homotopy category of bimodules over the coinvariant
algebra.

Let us first consider the case of the Lie algebra $\mathfrak{sl}_2$. In this
case, the coinvariant algebra is isomorphic to the algebra $D=\mathbb{C}[x]/(x^2)$
of dual numbers, graded such that $\deg(x)=2$. The identity functor $\theta_e$ is represented by the regular $D$-$D$-bimodule $D$ graded naturally. 
The functor $\theta_s$ is represented by the $D$-$D$-bimodule 
$D\otimes_{\mathbb{C}}D$, graded  such that $\deg(1\otimes 1)=-1$.
The adjunction morphism $\mathrm{adj}^s:\theta_s\to\theta_e\langle 1\rangle$ 
is given by the  multiplication map while the adjunction morphism
$\mathrm{adj}_s:\theta_e\langle -1\rangle \to\theta_s$ sends 
$1$ to $1\otimes x+x\otimes 1$. This means that the map $\alpha_s$
is the endomorphism of $D\otimes_{\mathbb{C}}D$ sending 
$1\otimes 1$ to $1\otimes x+x\otimes 1$. 
The two ways around the diagram in \eqref{eq7.4.1-3}
send $1\otimes 1$ to $x\otimes x\otimes 1+x\otimes 1\otimes x+1\otimes x\otimes x$.

The upper map of the diagram sends $1\otimes 1$ to 
$1\otimes 1\otimes x+1\otimes x\otimes 1$. The latter element generates
a summand of $\theta_s\theta_s$, call it $X$. We also have the 
corresponding summand of $\theta_s\theta_s\langle 2\rangle$, call it $Y$.
These two summands, together with the left part of the diagram, form a
subdiagram (that is they are closed with respect to all maps on the diagram).
By construction, the total complex of this subdiagram is homotopic to zero.

The image of the element $1\otimes 1\otimes 1$ under the right vertical map is
$x\otimes 1\otimes 1+ 1\otimes x\otimes 1$. When factoring out the (homotopic to
zero) subcomplex from the previous paragraph, we send 
$1\otimes 1\otimes x+1\otimes x\otimes 1$ to zero, implying that 
\begin{displaymath}
x\otimes 1\otimes 1+ 1\otimes x\otimes 1=
x\otimes 1\otimes 1- 1\otimes 1\otimes x. 
\end{displaymath}
This means that the total complex of \eqref{eq7.4.1-3} is represented
in the homotopy category of complexes by the following complex of 
$D$-$D$-bimodules:
\begin{equation}\label{eq7.4.21}
0\longrightarrow  D\otimes_{\mathbb{C}}D
\overset{\beta_s}{\longrightarrow}  D\otimes_{\mathbb{C}}D 
\longrightarrow 0,
\end{equation}
where the map $\beta_s$ sends $1\otimes 1$ to $x\otimes 1- 1\otimes x$.

The same computation for the RHS of \eqref{eq7.4.1-3} results 
in the complex
\begin{equation}\label{eq7.4.22}
0\longrightarrow  D\otimes_{\mathbb{C}}D
\overset{\gamma_s}{\longrightarrow}  D\otimes_{\mathbb{C}}D 
\longrightarrow 0,
\end{equation}
where the map $\gamma_s$ sends $1\otimes 1$ to $1\otimes x- x\otimes 1$.
Since $\beta_s=-\gamma_s$, the complexes in \eqref{eq7.4.21} and
\eqref{eq7.4.22} are isomorphic, for example, by choosing the identity map between the
left non-zero components and minus the identity map between the right 
non-zero components. This completes the proof in the case of the algebra
$\mathfrak{sl}_2$.

The general case is similar. One uses that the coinvariant algebra
$\mathtt{C}$ is a free module over the subalgebra $\mathtt{C}^s$
of $s$-invariants with the basis given by $1$ and the coroot 
corresponding to $s$ (which we can denote by $x$). Then all the above
arguments go through, mutatis mutandis.

\subsection{Proof: Relations~\ref{sbm-eq2} and \ref{sbm-eq3}}\label{s7.5}

We prove the relation given by Equation~\ref{sbm-eq3}. The proof of 
the relation given by Equation~\ref{sbm-eq2} 
is similar. We assume that $m_{s,t}$ is finite and even and 
need to prove that 
\begin{equation}\label{eq7.5.1-1}
\hat{\theta}_s\circ 
\underbrace{\mathcal{L}\mathsf{C}_t\circ\mathcal{L}\mathsf{C}_s\circ\dots\circ
\mathcal{L}\mathsf{C}_t}_{m_{s,t}-1\text{ factors}}
\cong \underbrace{\mathcal{L}\mathsf{C}_t\circ\mathcal{L}\mathsf{C}_s\circ\dots\circ
\mathcal{L}\mathsf{C}_t}_{m_{s,t}-1\text{ factors}}
\circ \hat{\theta}_s. 
\end{equation}
Let us take a look at the LHS of \eqref{eq7.5.1-1}.  Set
\begin{displaymath}
w=\underbrace{sts\dots t}_{m_{s,t}\text{ factors}}
=\underbrace{tst\dots s}_{m_{s,t}\text{ factors}}.
\end{displaymath}
From \eqref{eqaaa111}, it follows that  the LHS of \eqref{eq7.5.1-1}
is given by the total complex of the following complex:
\begin{equation}\label{eq7.5.1-3}
\xymatrix{
\theta_{sw}\ar[rr]^{\gamma_1}\ar[d]_{\gamma_2}&&
\theta_{w}\ar[d]^{\gamma_4}\langle 1\rangle\\
\theta_{sw}\langle 2\rangle\ar[rr]^{\gamma_3}&&\theta_{w}\langle 3\rangle\\
&&
},
\end{equation}
for some natural transformations $\gamma_i$, with $\gamma_1$ and $\gamma_3$ non-zero.
Furthermore, from \cite{Bac01}, see Subsection~\ref{s5.4}, it follows that 
$\dim\cP^{\mathbb{Z}}(\theta_{sw},\theta_{w}\langle 1\rangle)=1$.
The projective module $P_{sw}$ is a submodule of $P_{w}$, which implies
that the maps $\gamma_1$ and $\gamma_3$ are injections. 

By construction of $\alpha_s$, the map $\gamma_4$ factors through
$\theta_{ws}$. Again, \cite{Bac01}, see Subsection~\ref{s5.4}, it follows that 
$\dim\cP^{\mathbb{Z}}(\theta_{w},\theta_{ws}\langle 1\rangle)=1$ and
$\dim\cP^{\mathbb{Z}}(\theta_{ws},\theta_{w}\langle 1\rangle)=1$. This means that
$\gamma_4$ is unique up to scalar. Since  $\gamma_3$ is an injection, it follows
that commutativity of \eqref{eq7.5.1-3} defines the map $\gamma_2$ uniquely
(as soon as the scalar for $\gamma_4$ is fixed).

It remains to note that the RHS of \eqref{eq7.5.1-1} is also given by the
total complex of a commutative diagram of the form \eqref{eq7.5.1-3}
(possibly for different choices of the $\gamma_i$'s). However, we have just 
established that the only real choices in  \eqref{eq7.5.1-3} are scalars.
This implies \eqref{eq7.5.1-1}.

\subsection{Another candidate that fails}\label{s7.6}

There is another obvious candidate for categorification of 
the desingularization map. For a simple reflection $s$, consider 
the following morphism  $\xi'$ in 
$\mathcal{K}^b(\cP^\mathbb{Z})$, where $\theta_s$ is 
in the homological position $0$:
\begin{displaymath}
\xymatrix{
0\ar[r]\ar[d]&0\ar[r]\ar[d]&\theta_s\ar[r]^{\mathrm{adj}^s}\ar@{=}[d]
&\theta_e\langle 1\rangle\ar[r]\ar[d]&0\ar[d]\\
0\ar[r]&\theta_e\langle -1\rangle\ar[r]^{\mathrm{adj}_s}
&\theta_s\ar[r]&0\ar[r]&0
} 
\end{displaymath}
Setting $\beta_s=\mathrm{adj}^s\mathrm{adj}_s$, we obtain that the complex
\begin{equation}\label{eq7.6.1}
0\longrightarrow \theta_e\langle -1\rangle\overset{\beta_s}{\longrightarrow} 
\theta_e\langle 1\rangle \longrightarrow 0
\end{equation}
is a minimal representative for the cone of $\xi'$. 
Note that the identity map on
$\theta_s$ splits off as a direct summand of the cone and
hence disappears in this minimal representative. We denote 
the complex in \eqref{eq7.6.1} by $\check{\theta}_s$.
The following negative result shows that sending $\tau_s$ to
$\check{\theta}_s$ does not define a weak action of the singular braid
monoid.

\begin{proposition}\label{prop7.6-3}
For $s\in S$, we have $\check{\theta}_s\circ\mathcal{L}\mathsf{C}_s 
\not\cong \mathcal{L}\mathsf{C}_s\circ \check{\theta}_s$. 
\end{proposition}

\begin{proof}
It is enough to consider the case of the Lie algebra $\mathfrak{sl}_2$.
Again, we interpret all objects using bimodules over the dual numbers $D$.
Then the map $\beta_s$ is given by the bimodule homomorphism
$D\to D$ sending $1$ to $2x$. 

The composition $\check{\theta}_s\circ\mathcal{L}\mathsf{C}_s$ is then
given by the total complex of the following commutative diagram:
\begin{displaymath}
\xymatrix{
D\langle -2\rangle\ar[rr]^{\gamma_1}\ar[d]_{\gamma_2}&&
D\otimes_{\mathbb{C}}D\langle -1\rangle\ar[d]^{\gamma_4}\\
D\ar[rr]^{\gamma_3} &&D\otimes_{\mathbb{C}}D\langle 1\rangle\\
}
\end{displaymath}
Here
\begin{itemize}
\item $\gamma_1$ and $\gamma_3$ send $1$ to $1\otimes x+x\otimes 1$,
\item $\gamma_2$ sends $1$ to $2x$,
\item $\gamma_4$ sends $1\otimes 1$ to $2x\otimes 1$. 
\end{itemize}

In turn, the composition $\mathcal{L}\mathsf{C}_s\circ\check{\theta}_s$ is
given by the total complex of the following commutative diagram:
\begin{displaymath}
\xymatrix{
D\langle -2\rangle\ar[rr]^{\gamma_1}\ar[d]_{\gamma_2}&&
D\otimes_{\mathbb{C}}D\langle -1\rangle\ar[d]^{\gamma'_4}\\
D\ar[rr]^{\gamma_3} &&D\otimes_{\mathbb{C}}D\langle 1\rangle\\
}
\end{displaymath}
Here $\gamma_1$, $\gamma_2$ and $\gamma_3$ are as in the previous paragraph,
while $\gamma'_4$ sends $1\otimes 1$ to $1\otimes 2x$.

Since $\gamma_4$ and  $\gamma'_4$ are linearly independent, it follows that the
two total complexes are not isomorphic.
\end{proof}

\vspace{1cm}

\noindent
H.~J.: Department of Mathematics, Uppsala University, Box. 480,
SE-75106, Uppsala,\\ SWEDEN, email: {\tt helena.jonsson\symbol{64}math.uu.se}

\noindent
V.~M.: Department of Mathematics, Uppsala University, Box. 480,
SE-75106, Uppsala,\\ SWEDEN, email: {\tt mazor\symbol{64}math.uu.se}

\noindent
E.~P.~W.: Department of Mathematics, Uppsala University, Box. 480,
SE-75106,\\ Uppsala, SWEDEN, email: {\tt elin.persson.westin\symbol{64}math.uu.se}

\noindent
S.~S.: Department of Mathematics, Uppsala University, Box. 480,
SE-75106, Uppsala,\\ SWEDEN, email: {\tt shraddha.srivastava\symbol{64}math.uu.se}

\noindent
M.~S.: Department of Mathematics, Uppsala University, Box. 480,
SE-75106, Uppsala,\\ SWEDEN, email: {\tt mateusz.stroinski\symbol{64}math.uu.se}

\noindent
X.~Z.: School of Mathematical Sciences, Tongji University,
Shanghai 200092,\\ P.~R.~CHINA, email: {\tt 1810079\symbol{64}tongji.edu.cn}


\begin{thebibliography}{99999999}
\bibitem[An04]{An04} Antony, N. On the injectivity of the Vassiliev 
homomorphism of singular Artin monoids. Bull. Austral. Math. Soc. 
{\bf 70} (2004), no. 3, 401--422.
\bibitem[An05]{An05} Antony, N. On singular Artin monoids and 
contributions to Birman's conjecture. 
Comm. Algebra {\bf 33} (2005), no. 11, 4043--4056.
\bibitem[Bac92]{Bac01} Backelin, E. The Hom-spaces between projective 
functors. Represent. Theory {\bf 5} (2001), 267--283.
\bibitem[Bae92]{Ba92} Baez, J. Link invariants of finite type and 
perturbation theory. Lett. Math. Phys. {\bf 26} (1992), no. 1, 43--51. 
\bibitem[BB81]{BB} Beilinson, A.; Bernstein, J. Localisation de 
$\mathfrak{g}$-modules. C. R. Acad. Sci. Paris Ser. I Math. 
{\bf 292} (1981), no. 1, 15--18.
\bibitem[BG80]{BG} Bernstein, J.; Gelfand, S. Tensor products of finite- 
and infinite-dimensional representations of semisimple Lie algebras. 
Compositio Math. {\bf 41} (1980), no. 2, 245--285. 
\bibitem[BGG75]{BGG1} Bernstein, I.; Gelfand, I.; Gelfand, S. Differential 
operators on the base affine space and a study of $\mathfrak{g}$-modules. 
Lie groups and their representations (Proc. Summer School, Bolyai Janos 
Math. Soc., Budapest, 1971), pp. 21--64. Halsted, New York, 1975. 
\bibitem[BGG76]{BGG2}  Bernstein, I.; Gelfand, I.; Gelfand, S. 
A certain category of $\mathfrak{g}$-modules. (Russian) 
Funkcional. Anal. i Prilozen. {\bf 10} (1976), no. 2, 1--8. 
\bibitem[Bi93]{Bi93} Birman, J. New points of view in knot theory. 
Bull. Amer. Math. Soc. (N.S.) {\bf 28} (1993), no. 2, 253--287. 
\bibitem[Br37]{Br} Brauer, R. On algebras which are connected with 
the semisimple continuous groups. Ann. of Math. (2) {\bf 38} (1937), no. 4, 857--872. 
\bibitem[BK81]{BK} Brylinski, J.-L.; Kashiwara, M. Kazhdan-Lusztig conjecture 
and holonomic systems. Invent. Math. {\bf 64} (1981), no. 3, 387--410.
\bibitem[Ca86]{Ca} Carlin, K. Extensions of Verma modules. Trans. 
Amer. Math. Soc. {\bf 294} (1986), no. 1, 29--43.
\bibitem[CE22]{CE} Clark, C.; East, J.
Presentations for wreath products involving symmetric inverse
monoids and categories. Preprint arXiv:2204.06992.
\bibitem[C-Z18]{CDJMPZZ} Chen, C.-W.; Frisk Dubsky, B.; Jonsson, H.; 
Mazorchuk, V.; Persson Westin, E.; Zhang, X.; Zimmermann, J. 
Extreme representations of semirings. Serdica Math. J. {\bf 44} 
(2018), no. 3-4, 365--412.
\bibitem[CL15]{CL} Cohen, A.; Liu, S. Brauer algebras of type 
B. Forum Math. {\bf 27} (2015), no. 2, 1163--1202.
\bibitem[Co00]{Co00} Corran, R. A normal form for a class of monoids 
including the singular braid monoids. J. Algebra {\bf 223} 
(2000), no. 1, 256--282.
\bibitem[Co05]{Co05} Corran, R. Conjugacy in singular Artin monoids. 
J. Aust. Math. Soc. {\bf 79} (2005), no. 2, 183--212.
\bibitem[CM17]{CM} Coulembier, K.; Mazorchuk, V. Dualities and derived 
equivalences for category $\mathcal{O}$. Israel J. Math. {\bf 219} (2017), no. 2, 661--706.
\bibitem[EEF05]{EEF} Easdown, D.; East, J.; FitzGerald, D. 
Presentations of factorizable inverse monoids. Acta Sci. Math. (Szeged) 
{\bf 71} (2005), no. 3-4, 509--520.
\bibitem[Ea06a]{Ea06}  East, J. Birman's conjecture is true for $I_2(p)$.
J. Knot Theory Ramifications {\bf 15} (2006), no. 2, 167--177. 
\bibitem[Ea06b]{Ea06b}  East, J. The factorizable braid monoid. 
Proc. Edinb. Math. Soc. (2) {\bf 49} (2006), no. 3, 609--636.
\bibitem[Ea15]{Ea15}  East, J. Singular braids and partial permutations. 
Acta Sci. Math. (Szeged) {\bf 81} (2015), no. 1-2, 55--77.
\bibitem[EF13]{EF13}  Everitt, B.; Fountain, J.
Partial mirror symmetry, lattice presentations and algebraic monoids. 
Proc. Lond. Math. Soc. (3) {\bf 107} (2013), no. 2, 414--450. 
\bibitem[Fi03]{Fi03} FitzGerald, D. G. A presentation for the monoid of 
uniform block permutations. Bull. Austral. Math. Soc. {\bf 68} (2003), no. 2, 317--324. 
\bibitem[FL06]{FL06} FitzGerald, D.; Leech, J. Dual symmetric inverse 
monoids and representation theory. J. Austral. Math. Soc. Ser. 
A {\bf 64} (1998), no. 3, 345--367.
\bibitem[GM93]{GM2} Ganyushkin, O.; Mazorchuk, V. Factor powers of semigroups of 
transformations. (Ukrainian) Dopov./Dokl. Akad. Nauk Ukraini {\bf 1993}, no. 12, 5--9 (1994).
\bibitem[GM95a]{GM3} Ganyushkin, O.; Mazorchuk, V. Factor powers of finite symmetric groups. 
(Russian) Mat. Zametki {\bf 58} (1995), no. 2, 176--188; translation in Math. Notes 
{\bf 58} (1995), no. 1-2, 794--802 (1996) 
\bibitem[GM95b]{GM4} Ganyushkin, O.; Mazorchuk, V. The structure of subsemigroups of 
factor powers of finite symmetric groups. (Russian) Mat. Zametki {\bf 58} 
(1995), no. 3, 341--354, 478; translation in Math. Notes {\bf 58} (1995), no. 3-4, 910--920 (1996)
\bibitem[GM03]{GM5} Ganyushkin, O.; Mazorchuk, V. On the radical of $\mathcal{FP}^+(S_n)$. 
Mat. Stud. {\bf 20} (2003), no. 1, 17--26.
\bibitem[GM09]{GM} Ganyushkin, O.; Mazorchuk, V. Classical finite transformation semigroups. 
An introduction. Algebra and Applications, {\bf 9}. Springer-Verlag London, Ltd., London, 2009. 
\bibitem[GP05]{GP05} Godelle, E.; Paris, L. On singular Artin monoids. 
Geometric methods in group theory, 43--57, Contemp. Math., {\bf 372}, 
Amer. Math. Soc., Providence, RI, 2005.
\bibitem[Hu10]{Hu10} Hudzenko, S. V. Automorphisms of a finitary factor power 
of an infinite symmetric group. Ukrainian Math. J. {\bf 62} (2010), no. 7, 1158--1162; 
translated from Ukraïn. Mat. Zh. {\bf 62} (2010), no. 7, 997--1001.
\bibitem[Hu72]{Hu1} Humphreys, J. Introduction to Lie algebras and 
representation theory. Graduate Texts in Mathematics, Vol. {\bf 9}. 
Springer-Verlag, New York-Berlin, 1972. xii+169 pp.
\bibitem[Hu08]{Hu2} Humphreys, J. Representations of semisimple Lie 
algebras in the BGG category $\mathcal{O}$. Graduate Studies in 
Mathematics, {\bf 94}. American Mathematical Society, 
Providence, RI, 2008. xvi+289 pp. 
\bibitem[Ja17]{Ja17} Jab{\l}onowski, M. Presentations and representations 
of surface singular braid monoids. J. Korean Math. Soc. {\bf 54} 
(2017), no. 3, 749--762.
\bibitem[KL79]{KL} Kazhdan, D.; Lusztig, G. Representations of Coxeter 
groups and Hecke algebras. Invent. Math. {\bf 53} (1979), no. 2, 165--184.
\bibitem[KS02]{KS02} Khovanov, M.; Seidel, P. Quivers, Floer cohomology, 
and braid group actions. J. Amer. Math. Soc. {\bf 15} (2002), no. 1, 203--271. 
\bibitem[Ko06]{Ko} Kosuda, M. Irreducible representations of the party algebra. 
Osaka J. Math. {\bf 43} (2006), no. 2, 431--474. 
\bibitem[KM11]{KuMa} Kudryavtseva, G.; Maltcev, V. Two generalisations of 
the symmetric inverse semigroups. Publ. Math. Debrecen {\bf 78} 
(2011), no. 2, 253--282.
\bibitem[KMU15]{KMU} Kudryavtseva, G.; Maltcev, V.; Umar, A. 
Presentation for the partial dual symmetric inverse monoid. 
Comm. Algebra {\bf 43} (2015), no. 4, 1621--1639.
\bibitem[KMaz06]{KuMaz} Kudryavtseva, G.; Mazorchuk, V. On presentations of 
Brauer-type monoids. Cent. Eur. J. Math. {\bf 4} (2006), no. 3, 413--434.
\bibitem[Mar94]{Mar94} Martin, P., Temperley-Lieb algebras for nonplanar 
statistical mechanics~-- the partition algebra construction, J. 
Knot Theory Ramifications {\bf 3} (1994), 51--82. 
\bibitem[MM13]{MM13} Martin, P.; Mazorchuk, V. Partitioned binary relations. 
Math. Scand. {\bf 113} (2013), no. 1, 30--52. 
\bibitem[MM14]{MM14} Martin, P.; Mazorchuk, V. On the representation theory 
of partial Brauer algebras. Q. J. Math. {\bf 65} (2014), no. 1, 225--247.
\bibitem[Ma98]{Ma98} Mazorchuk, V. On the structure of Brauer semigroup 
and its partial analogue, Problems in Algebra, {\bf 13} (1998), 29--45.
\bibitem[Ma00]{Ma00} Mazorchuk, V. All automorphisms of $\mathcal{FP}^+(S_n)$ 
are inner. Semigroup Forum {\bf 60} (2000), no. 3, 486--490.
\bibitem[Ma07]{Ma07} Mazorchuk, V. Some homological properties of 
the category $\mathcal{O}$. Pacific J. Math. {\bf 232} (2007), no. 2, 313--341.
\bibitem[Ma09-1]{Ma} Mazorchuk, V. Applications of the category of 
linear complexes of tilting modules associated with the category $\mathcal{O}$. 
Algebr. Represent. Theory {\bf 12} (2009), no. 6, 489--512.
\bibitem[Ma09-2]{Ma09} Mazorchuk, V. Simple modules over factorpowers.
Acta Sci. Math. (Szeged) {\bf 75} (2009), no. 3-4, 467--485. 
\bibitem[MaSr21]{MaSr} Mazorchuk, V.; Srivastava, S.
Jucys-Murphy elements and Grothendieck groups for generalized rook monoids.
Preprint arXiv:2104.13632. To appear in the Journal of Combinatorial Algebra.
\bibitem[MaSr22]{MaSr22} Mazorchuk, V.; Srivastava, S.
Multiparameter colored partition category and the product of reduced Kronecker coefficients.
Preprint arXiv:2204.13564. 
\bibitem[MaSte12]{MaSt} Mazorchuk, V.; Steinberg, B. Double Catalan monoids. 
J. Algebraic Combin. {\bf 36} (2012), no. 3, 333--354.
\bibitem[MaStr05]{MaStr} Mazorchuk, V.; Stroppel, C. Translation and 
shuffling of projectively presentable modules and a categorification 
of a parabolic Hecke module. Trans. Amer. Math. Soc. {\bf 357} (2005), no. 7, 2939--2973. 
\bibitem[MaStr07]{MaStr2} Mazorchuk, V.; Stroppel, C. 
On functors associated to a simple root. J. Algebra {\bf 314} (2007), no. 1, 97--128.
\bibitem[OSSZ21]{OSSZ21} Orellana, R.; Saliola, F.; Schilling, A.; Zabrocki, M.
Plethysm and the algebra of uniform block permutations.
Preprint arXiv:2112.13909.
\bibitem[Pa04]{Pa04} Paris, L. The proof of Birman's conjecture 
on singular braid monoids. Geom. Topol. {\bf 8} (2004), 1281--1300.
\bibitem[PW70]{PW70} Plemmons, R.; West, M. On the semigroup of binary relations. 
Pacific J. Math. {\bf 35} (1970), 743--753.
\bibitem[Ro06]{Ro06} Rouquier, R. Categorification of $\mathrm{sl}_2$ and braid groups.
Trends in representation theory of algebras and related topics, 137--167, 
Contemp. Math., {\bf 406}, Amer. Math. Soc., Providence, RI, 2006. 
\bibitem[So90]{So} Soergel, W. Kategorie $\mathcal{O}$, perverse Garben und 
Moduln {\"u}ber den Koinvarianten zur Weylgruppe. 
J. Amer. Math. Soc. {\bf 3} (1990), no. 2, 421--445.
\bibitem[So07]{So2} Soergel, W. Kazhdan-Lusztig-Polynome und unzerlegbare 
Bimoduln {\"u}ber Polynomringen. J. Inst. Math. Jussieu {\bf 6} (2007), no. 3, 501--525. 
\bibitem[St03]{St} Stroppel, C. Category $\mathcal{O}$: gradings and translation 
functors. J. Algebra {\bf 268} (2003), no. 1, 301--326.
\bibitem[Ve10]{Ve10} Vershinin, V. On the singular braid monoid. 
St. Petersburg Math. J. {\bf 21} (2010), no. 5, 693--704.
\end{thebibliography}
\end{document}